\documentclass[reqno]{amsart}
\usepackage{amsmath,amsfonts,amssymb,amsthm}
\usepackage{mathtools}
\usepackage{commath}
\usepackage{cleveref}
\usepackage{xfrac}
\usepackage{nicefrac}
\usepackage{enumitem}
\let\norm\undefined % <-- "Undefine" \norm
\DeclarePairedDelimiter\norm{\lVert}{\rVert}

\numberwithin{equation}{section}

\DeclareMathOperator{\conv}{conv}
\DeclareMathOperator{\dist}{dist}
\DeclareMathOperator{\ran}{ran}
\DeclareMathOperator{\supp}{supp}
\DeclareMathOperator{\maxsupp}{maxsupp}
\DeclareMathOperator{\minsupp}{minsupp}
\newcommand{\andd}{\,\,\,\text{and}\,\,}
\newcommand{\N}{\mathbb{N}}
\newcommand{\Q}{\mathbb{Q}}
\newcommand{\e}{\epsilon}
\newcommand{\mc}{\mathcal}

\newcommand{\tree}[1][A]{\mathcal{#1}}
\newtheorem{theorem}{Theorem}[section]
\newtheorem{lemma}[theorem]{Lemma}
\newtheorem{remark}[theorem]{Remark}
\newtheorem{remarks}[theorem]{Remarks}
\newtheorem{definition}[theorem]{Definition}
\newtheorem{proposition}[theorem]{Proposition}
\newtheorem{corollary}[theorem]{Corollary}
\newtheorem{notation}[theorem]{Notation}
\newtheorem*{claim}{\textit{Claim}}
\newtheorem{problem}{Problem}

\newcommand{\wmt}{W_{mT}}
\newcommand{\bxk}[1][x]{ (#1_{k})_{k\in\N}
}

\begin{document}

%%
%% The title of the paper goes here.  Edit to your title.
%%

\title{The HI extension of the standard HI spaces}

%%
%% Now edit the following to give your name and address:
%% 

\author{Spiros A. Argyros}
\address{National Technical University of Athens, Faculty of Applied Sciences, Department of Mathematics, Zografou Campus, 157 80, Athens, Greece.}
\email{sargyros@math.ntua.gr}
%\urladdr{www.math.sc.edu/$\sim$howard} % Delete if not wanted.

%%
%% If there is another author uncomment and edit the following.
%%

\author{Antonis Manoussakis}
\address{School of Environmental Engineering, Technical University of Crete, University Campus, 73100 Chania, Greece}
\email{amanousakis@isc.tuc.gr}
% \urladdr{www.math.sc.edu/$\sim$second}
\author{Pavlos Motakis }
\address{Department of Mathematics and Statistics, York University, 4700 Keele Street, Toronto, Ontario, M3J 1P3, Canada}
\email{pmotakis@yorku.ca}

%%
%% If there are three of more authors they are added in the obvious
%% way. 
%%

%%%
%%% The following is for the abstract.  The abstract is optional and
%%% if not used just delete, or comment out, the following.
%%%

\begin{abstract}
A Hereditarily Indecomposable (HI)  Banach space $X$ admits an HI
extension if there exists an HI space $Z$ such that $X$ is isomorphic
to a subspace $Y$ of $Z$ and $Z/Y$ is of infinite dimension. The
problem whether or not every HI  space admits an HI extension is attributed to A. Pelczynski. In this paper we present a method to define
HI-extensions of the standard HI spaces, a class which includes
the Gowers-Maurey space, asymptotic $\ell_{p}$-HI spaces and others.
\end{abstract}
 \maketitle

%%
%% LaTeX can automatically make a table of contents.  This is done by
%% uncommenting the following:
%%

%\tableofcontents
 \section{Introduction}\label{sintro}
 The problem of the HI extension of an HI Banach space  can be considered as
 part of a general problem  which concerns the embedding of Banach
 spaces into a space with tight structure. There are several results
 related to this problem with the more representative being the following:
 The first one stated and proved  in \cite{ARaik} asserts that every
 separable reflexive Banach space is isomorphic to a  subspace of a
 reflexive indecomposable Banach space. The second one is similar to
 the previous result where the superspace admits very few operators. In
 particular every separable reflexive Banach space is isomorphic to a
 subspace of a $\mc{L}_{\infty}$ space which satisfies the
 ``scalar-plus-compact'' property. This result is presented in
 \cite{AFHORSZ} for uniformly convex spaces and the general case is
 proved in an unpublished paper of the same authors.
In this direction it remains open whether every separable Banach space
not containing $c_{0}$ is embedded  into a $\mc{L}_{\infty}$ space with
very few operators.

Returning to the problem of the HI  extension
of an HI space, attributed to A. Pelczynski, we point out that this is a
delicate and hard problem. This is mainly due to the quasi-minimality
of HI Banach spaces \cite{G}. To explain the obstacle a little further, assume
that $X$ is an HI space and  $Z$ is also  HI containing isomorphically
the space $X$. Then every subspace $Y$ of $X$ and every subspace $E$
of $Z$
contain further subspaces  $D_{1}$  of $Y$ and $D_{2}$ of $E$ which are
isomorphic. This means that   for a possible construction of an HI
superspace $Z$ of $X$ we ought to extract the structure of the space $X$
in the space $Z$ and this is hard,  in particular when we do not have a
precise description of the structure of the space $X$.

As it stands, there exists a specific type of HI space for which the answer to Pelczynski’s problem is known to be affirmative. In \cite{CFG}, J. M. F. Castillo, V. Ferenczi and M. Gonzalez proved that Ferenczi's uniformly convex HI space $\mathcal{F}$ from \cite{F} admits an HI extension $\mathcal{F}_{2}$. The space $\mathcal{F}$ is the outcome of a complex interpolation scheme, and this is an unavoidable component of the definition of $\mathcal{F}_{2}$. As such, the method from \cite{CFG} is not applicable to the more standard HI spaces defined via the method of norming sets. The aim of this paper is to exhibit a method that permits us to define HI extensions of the latter class of HI spaces, which includes the first generation of HI spaces such as the Gowers-Maurey HI space, \cite{GM}, the Argyros-Deliyanni
asymptotic $\ell_{1}$ HI space, \cite{AD}, and the Deliyanni-Manoussakis asymptotic $\ell_{p}$-HI
spaces, \cite{DM}. We have  chosen to apply the method on a variant
of  Gowers-Maurey space, denoted by $X_{hi}$, with its norming set
$W_{hi}$ being a subset of the norming set of the mixed Tsirelson space
$T[(\mc{A}_{n_{j}}, m_{j}^{-1})_{j\in\N}]$.

The norming set of any HI space which has been defined in the
literature  has two  inevitable  characteristics. The first  one is
some  unconditional structure which is usually  imposed by a  mixed
Tsirelson space.  For example in Gowers-Maurey space, Schlumprecht
space \cite{Sch} plays that role. The second one is the conditional
structure which is responsible  for the tight structure of the HI
space and it is traced in the earlier work of B. Maurey and
H. P. Rosenthal \cite{MR}. In the case of $X_{hi}$ the norming set
$W_{hi}$ is defined as follows: The set $W_{hi}$ is a subset of
$c_{00}(\N)$ which is minimal with respect to the following properties.
\begin{enumerate}[leftmargin=21pt,label=(\roman*)]

\item  It is rationally convex, symmetric, closed on the projections on
intervals  of $\N$ and contains $\pm e_{n}^{*}$, $n\in\N$.

\item It is closed in all $(\mc{A}_{n_{2j}},m_{2j}^{-1})$
operations. That means that for every  family $(f_{i})_{i=1}^{n_{2j}}$
of elements of $W_{hi}$ with successive supports  the functional
$f=m_{2j}^{-1}\sum_{i=1}^{n_{2j}}f_{i}\in W_{hi}$. In this case we say that  $f$ is a weighted functional with weight $w(f)$ of $f$ equal to $m_{2j}$.( Note that this is the
unconditional structure mentioned above).

\item It includes all $f=m_{2j-1}^{-1}\sum_{i=1}^{n_{2j-1}}f_{i}$ where
$(f_{i})_{i=1}^{n_{2j-1}}$ is a $(2j-1)$-special sequence. For such a functional we set  $w(f)=m_{2j-1}$.

\end{enumerate}

$(2j-1)$-special sequences are sequences $(f_{i})_{i=1}^{n_{2j-1}}$ of
successive weighted functionals which are defined through a coding
function. Their important property is a tree-like property, namely
for every $i>1$ the weight $w(f_{i})=m_{2j_{i}}$ of $f_{i}$   uniquely determines the previous elements of the sequence.
For a precise  definition of the $(2j-1)$-special sequences see
Definition~\ref{sseq}. This type of functionals imposes the HI
structure of the space $X_{hi}$.

We now pass to describe the HI extension of the space $X_{hi}$ mentioned above. Our approach is the natural one. Namely, we shall define a new norming set $W_{ex}$ such that the induced norm on $c_{00}(\mathbb{N})$ is an HI one and also the completion of the subspace generated by the even coordinates $(e_{2n})_{n\in\mathbb{N}}$ is isometric to the initial space $X_{hi}$. For this, we let $\widehat W_{hi} =  \{\hat g: g\in W_{hi}\}$ where for $g\in W_{hi}$ and $n\in\mathbb{N}$, $\hat g(2n) = g(n)$ and $\hat g(2n-1) = 0$. Then, $W_{ex}$ is the norming set that satisfies the following two conditions:
\begin{enumerate}[leftmargin=17pt,label=(\alph*)]

\item The norm induced by $W_{ex}$ is an HI one.

\item $W_{ex}|_{\mathbb{N}_{ev}} = \widehat W_{hi}$, where $\mathbb{N}_{ev}$ is the set of even positive integers and $g|_{\mathbb{N}_{ev}} = g\cdot 1_{\mathbb{N}_{ev}}$.

\end{enumerate}

We set   $X_{ex}$ the completion of   $ (c_{00}(\mathbb{N}),\|\cdot\|_{W_{ex}})$.
Condition (b) yields that the subspace  $X_{ev} = \overline{\langle e_{2n}:n\in\mathbb{N}\rangle}$   of  $X_{ex}$  is isometric to  $X_{hi}$. In the sequel we denote the subspace $X_{ev}$  by  $\widehat X_{hi}$. Let us also observe that since the space $X_{ex} $ is an HI space, the set $\widehat W_{hi}$ cannot be a subset of $W_{ex}$. Otherwise, the subspace  $\widehat X_{hi}$ would be complemented in $X_{ex}$. Moreover, since $\|\cdot\|_{W_{ex}}$ is HI, the norming set  $W_{ex}$ must include  special functionals which are entirely responsible for the HI structure. At the same time, for every special functional $f\in W_{ex}$, $f|_{\mathbb{N}_{ev}} = \hat g$ with $g\in W_{hi}$. This means that the added special functional in $W_{ex}$ does not contribute to  $\widehat X_{hi}$  norm beyond the elements of $\widehat W_{hi}$. To satisfy the previously explained two necessary requirements, we introduce paired special functionals that  will be described below. The norming set $W_{ex}$ is defined to satisfy the following  properties.
\begin{enumerate}[leftmargin=21pt,label=(\roman*)]

\item It contains $(e_n^*)_{n\in\mathbb{N}}$, it is symmetric, closed under rational convex combinations, and the restrictions on intervals of $\mathbb{N}$.

\item It is closed under the even operations $(\mathcal{A}_{n_{2j}},m_{2j}^{-1})_{j}$.

\item $W_{ex}|_{\mathbb{N}_{ev}} = \widehat W_{hi}$.

\item It contains the paired special functionals.
\end{enumerate}

The paired special functionals are of the form $f = m_{2j+1}^{-1}\sum_{k=1}^{n_{2j+1}}(f_k+\hat g_k)$ such that $\supp f_k\subset \mathbb{N}_{od}$, the set of the odd positive integers, $w(f_k) = w(g_k)$, and $m_{2j+1}^{-1}\sum_{k=1}^{n_{2j+1}}g_k$ is a special functional in the norming set $W_{hi}$. The definition of the paired special sequence $f_1,\hat g_1,\ldots,f_{n_{2j+1}},\hat g_{n_{2j+1}}$ uses a new coding function so that the tree-like property is also satisfied. Notice that the paired special functionals satisfy the requirement that $f|_{\mathbb{N}_{ev}} = \hat g$ with $g\in W_{hi}$. The HI property of the space $X_{ex}$ follows from the following criterion from \cite[Lemma 7.1]{CFG} (see Proposition \ref{hi-1}): the subspace $\widehat X_{hi}$ is HI and the quotient map  $Q: X_{ex}\to X_{ex}/ \widehat X_{hi}$ is a strictly singular operator. The latter is a consequence of the next proposition (Proposition \ref{hi-2}).

\begin{proposition}
For every infinite dimensional subspace $Y$ of $X_{ex}$ we have that  $\mathrm{dist}(S_Y,S_{\widehat X_{hi}}) = 0$.
\end{proposition}
The key ingredients for proving this result are paired special functionals.

The paper is organized into nine sections.The first three sections, after the introduction, are devoted to a brief presentation of basic ingredients and results for defining the norming set  $W_{hi}$ and proving its properties. The next section contains the definition of the norming set $W_{ex}$ of the space $X_{ex}$. In the following two sections  the proof that the space $X_{ex}$ is HI is given. In 
the last two sections we present two further HI extensions. The first one concerns a norming set $W$ which is not closed in rational convex combinations and the second one refers to HI extensions of asymptotic $\ell_1$ HI spaces.

\section{Mixed Tsirelson spaces}\label{smt}
In this section we recall the definition of a class of  mixed Tsirelson spaces and some of their properties that we shall use. We refer to \cite{ATod} for a detailed  presentation. Also we follow standard terminology as in  \cite{AlK},\cite{LT}.

Let  $(m_{j})_{j=0}^{+\infty}$, $(n_{j})_{j=0}^{+\infty}$ be two strictly increasing sequences
satisfying the  following:
\begin{enumerate}
\item  $m_{0}=2$ and $m_{j+1}=m_{j}^{5}$.
\item $n_{0}=4$ and $ n_{j+1}=(12n_{j})^{s_{j}}$
  where $2^{s_{j}}\geq m_{j+1}^{3}$.
\end{enumerate}
By $\mc{A}_{n}$ we denote the family of the finite subsets of $\N$
with at most $n$ elements, i.e.,  $\mc{A}_{n}=\{F\subset\N:\#F\leq n\}$.

\begin{definition}
1) Let  $n\in\N$. A sequence   $(f_{i})_{i=1}^{k}$, $k\leq n$, of elements  of
$c_{00}(\N)$ is said to be $\mc{A}_{n}$-admissible if
$\supp (f_{1})<\supp (f_{2})<\dots<\supp (f_{k})$.

2) The $(\mc{A}_{n_{j}}, m_{j}^{-1})$-operation on $c_{00}(\N)$ 
is the operation which assigns to each
$n_{j}$-admissible sequence $(f_{i})_{i=1}^{k}$ the vector $m_{j}^{-1}(f_{1}+f_{2}+\dots+f_{k})$.
\end{definition}
For $f\in c_{00}(\N)$ we shall write  $w(f)=m_{j}$ if $f$ is the result of an
$(n_{j},m_{j}^{-1})$- operation.

We present the norming set of  the mixed Tsirelson space defined by the
families  $\mathcal{A}_{n_{j}}$ and the ``weights'' $(m_{j})_{j}$ which
is  a discretization of the norming set of  Schlumprecht space
$\mathcal{S}$, \cite{Sch}.

\begin{definition}
  Let  $W_{mT}$ be the smallest subset of $c_{00}(\N)$  with the following properties
\begin{enumerate}
\item $\pm e^{*}_{n} \in W_{mT}$ for every $n\in\N$.
\item  $W_{mT}$ is closed in the $(\mc{A}_{n_{j}},m_{j}^{-1})$-operations, i.e., for  every $\mc{A}_{n_{j}}$-admissible  sequence
  $(f_{i})_{i=1}^{k}$ of elements of $W_{mT}$ we have that
  $m_{j}^{-1}\sum_{i=1}^{k}f_{i}\in W_{mT}$.
\item  $W_{mT}$ is symmetric
and  closed under the restriction of its elements to intervals of $\N$.
  \item  $W_{mT}$ is closed in rational convex combinations.
  \end{enumerate}
\end{definition}
The set $W_{mT}$ can  also be defined inductively as $W_{mT}=\cup_{n=0}^{\infty}W_{mT}^{n}$ where  $W_{mT}^{0}=\{\pm e^{*}_{n}:n\in\N\}$  and 
$W_{mT}^{n}$ is the smallest subset of $c_{00}(\N)$ that contains $W_{mT}^{n-1}$ and satisfies properties (2)-(3) of the above definition with respect to  the functionals of $W_{mT}^{n-1}$.

We denote  by $X_{mT}=T[(\mc{A}_{n_{j}}, m_{j}^{-1})_{j}]$ the completion of  $c_{00}(\N)$  under the norm
$\norm{\cdot}_{W_{mT}}$
induced  by the set   $W_{mT}$, i.e.,
\[
  \norm{x}_{W_{mT}}=\sup\{f(x):f\in W_{mT}  \}
\]
\begin{definition}
Let $f\in W_{mT}$. For a finite tree $\tree$, a family $(f_{\alpha})_{\alpha\in\tree}$ is
said to be a tree analysis of $f$  if the following are satisfied.
\begin{enumerate}
\item   $\tree$ has a unique root denoted by 0 and $f_{0}=f$.
  \item  For every maximal node $\alpha\in\tree$ we have that $f_{\alpha}=\pm
    e^{*}_{n_{\alpha}}$.
  \item Let $\alpha$ be a non-maximal node of $\tree$
    and denote by $S_{\alpha}$ the  set of the immediate successors of $\alpha$. Then
     either  there exists  $j\in\N$ such that 
    $(f_{\beta})_{\beta\in S_{\alpha}}$  is a $n_{j}$-admissible subset of $W_{mT}$
    and $f_{\alpha}= m_{j}^{-1}\sum_{\beta\in S_{\alpha}}f_{\beta}$
 or there exists a rational convex combination $(r_{\beta})_{\beta\in
        S_{a}}$ such that  $f_{\alpha}=\sum_{\beta\in S_{\alpha}}r_{\beta}f_{\beta}$.
  \end{enumerate}
\end{definition}
\begin{proposition}
 Every $f\in W_{mT}$ admits a tree analysis.
\end{proposition}
The proof is easy.  We simple consider  all $f\in W_{mT}$ admitting a
tree analysis and then we show that this set satisfies (1),(2) and (3)
of the above definition.  The result follows from the minimality of
the set $W_{mT}$.

\begin{definition}
  Let $f\in\wmt$  and $(f_{\alpha})_{\alpha\in\tree}$ be a  tree   analysis  of $f$.
  For  $\alpha\in\tree$  the functional $f_{\alpha}$
 is said to be weighted  with weight $w(f_{\alpha})=m_{j}$ if $f_{\alpha}$ is the result
 of an $(\mc{A}_{n_{j}},m_{j}^{-1})$-operation of its successors.
\end{definition}

It was  proved in  \cite{ATolmemoir}, Lemma 3.15,  that every  $f\in
W_{mT}$ admits a tree analysis $(f_{\alpha})_{\alpha\in\tree}$ such that whenever  $f_{\alpha}$ is rational convex combination of its successors,
i.e., $f_{\alpha}=\sum_{\beta\in S_{\alpha}}r_{\beta}f_{\beta}$,  then
every $f_{\beta}$ is a weighted functional  and $\ran(f_{\beta})\subset\ran(f_{\alpha})$.

\section{Even operation complete subsets of $W_{mT}$}\label{eocomplete}
\begin{definition}
  A subset  $W$ of $W_{mT}$ is said to be  even operation complete  (EO-complete)  if
  \begin{enumerate}
  \item $\pm e^{*}_{n} \in W$ for every $n\in\N$.
    \item $W$  is closed   in the  $(\mc{A}_{n_{2j}},m_{2j}^{-1})$-operation
      for every $j\in\N$.
    \item  $W$ is symmetric 
      and closed under the restriction of its elements to intervals of $\N$.
\item  Every $f\in W$  admits a tree analysis $(f_{\alpha})_{\alpha\in\tree}$ such
  that $f_{\alpha}\in W$  for every $\alpha\in\tree$.
  \end{enumerate}
\end{definition}
In the rest  of this section by $W$ we will denote any EO-complete subset
of $W_{mT}$.
We will denote by $X_{W}$ the completion of $c_{00}(\N)$ under the
norm $\norm{\cdot}_{W}$ induced by $W$.

\begin{definition}
A vector $x\in X_{W}$ is said to be a $C-\ell_{1}^{k}$- average if  there exists
$x_{1}<x_{2}<\dots<x_{k}$  such  that
$x=\frac{1}{k}\sum_{i=1}^{k}x_{i}$, and  $\norm{x_{i}}\leq C\norm{x}$ for
each $i$.

Moreover  if $\norm{x}=1$  then $x$ is called   a normalized $C-\ell_{1}^{k}$-average.
\end{definition}
\begin{lemma}\label{l1averages}
Let $Y$   be a block subspace  of $X_{W}$. Then for every $j\in\N$
there exists a   block sequence  $(y_{i})_{i=1}^{n_{2j}}$ such that
$y_{i}\in Y$   and $n_{2j}^{-1}\sum_{i=1}^{n_{2j}}y_{i}$ is a normalized $2-\ell_{1}^{n_{2j}}$-average
\end{lemma}
The existence of normalized $\ell_{1}^{k}$-averages was  first established by
Th. Schlumprecht  in his space $S$ (\cite{Sch}), with the use  
of  Krivine theorem, since  $\ell_{1}$ is  finitely block representable  in
$S$. Later  W. T. Gowers and B. Maurey, (\cite{GM} Lemma 3),  provided  a  James type argument to construct
$2-\ell_{1}^{k}$-averages in $S$. For a proof  
in the setting of the  space $X_{W}$ we refer the reader to  \cite{ATod} Lemma II.22, \cite{AManStudia} Lemma 4.6

The following lemma  concerning  $\ell_{1}^{n{j}}$-averages is an
essential tool
for providing upper bounds  on the norm of certain vectors.
\begin{lemma}\label{i<j}
 Let $x$ be an $C-\ell_{1}^{n_{j}}$-average  in $X_{W}$  and $f\in W$
 with $w(f)=m_{i}$, $i<j$.  Then
 \[\abs{f(x)}\leq\frac{3C}{2m_{i}}. \]
\end{lemma}
We refer   to \cite{GM} Lemma 4,  or \cite{ATod} Lemma II.23, for a proof.

\begin{definition}\label{defRIS}
A block sequence  $\bxk$   in $X_{W}$ is said to be a $(C,\e)$-rapidly
increasing sequence (RIS), if
there exists a strictly
 increasing sequence  $(j_{k})_{k}$  of positive integers such that
\begin{enumerate}
\item $\norm{x_{k}}\leq C$ for all  $k$.
 \item    $\frac{1}{m_{j_{k+1}}} \cdot \#\supp(x_{k})\leq \e$ for every $k$.
 \item For every $k=1,2,\dots $  and every $f\in W$ with $w(f)=m_{i}, i< j_{k}$ we have
    that   $|f (x_{k} )|\leq \sfrac{C}{m_{i}}$.
 \end{enumerate}
\end{definition}

\begin{proposition}
Let $Y$
be a block subspace  of $X_{W}$. Then for  $\e>0$
there exists a normalized $(3,\e)$-RIS
sequence $(x_{k})_{k}$  with $x_{k}\in Y$  for every $k$.
\end{proposition}

Indeed, Lemma~\ref{l1averages}   yields that  we can get a block
sequence  $(x_{j})_{j}$ such that  $x_{j}$ is a normalized
2-$\ell_{1}^{n_{2j}}$-average and $x_{j}\in Y$ for every $k$.  Choose  a
subsequence  $(x_{j_{k}})_{k}$   such that    $\#\supp(x_{j_{k}})\leq\e m_{2j_{k+1}}$.
Lemma~\ref{i<j}    yields that    $(x_{j_{k}})_{k}$ is  a  $(3,\e)$-RIS.
\subsection{The auxiliary space}
In order to get  upper bounds  on  the action of  weighted functionals on certain vectors we shall use an auxiliary space to reduce the high complexity of the calculations (see subsection~\ref{basicin}).
The auxiliary  space
$X_{aux}$   is the mixed  Tsirelson space
$T[(\mc{A}_{4n_{j}},m_{j}^{-1})_{j\in\N}]$.

The result we  shall  use from the auxiliary space  are 
the  upper 
estimations
of the functionals of its  norming set  $W_{aux}$ on the averages of its  basis,
and in particular  the following lemma.
\begin{lemma}\label{basisaverages}
Let $x=\frac{m_{j}}{n_{j}}\sum_{k=1}^{n_{j}}e_{k}\in X_{aux}$. Then  for 
every $f\in W_{aux}$ with  $w(f)=m_{i}$,
\[
  \abs{f(x)}\leq\begin{cases}\frac{2}{m_{i}}\,\,\text{if $i<j$}\\
        \frac{m_{j}}{m_{i}}\,\,\text{if $i\geq j $}.
  \end{cases}
\]
Moreover  if $f$  has a tree analysis $(f_{\alpha})_{\alpha\in\tree}$ such
that for every weighted  $f_{\alpha}$, $w(f_{\alpha})\ne m_{j}$  then
$\abs{f(x)}\leq \frac{2}{m^{2}_{j}}$.
\end{lemma}

For a proof of this result see \cite{ATod},Lemma II.9,  (or
\cite{AManStudia},  Lemma 4.2).
\begin{remark}\label{basisaverages2}
Let us note  that the above estimations hold also   for
$n_{j}$-averages of the basis  of the space $X_{W}$.
\end{remark}
\subsection{The basic inequality}\label{basicin}
The basic inequality is the fundamental  tool which permits us to  show
that the action of a
weighted functional $f\in W$  on a  $n_{j}$-average  of a RIS  admits
similar upper bound as the action of a functional  
 $g$ in $W_{aux}$ with the same weight  with $f$,  acting  on a
 $n_{j}$-average of the basis  
 of $X_{aux}$.  In particular the following holds.
 \begin{proposition}
   Let  $(x_{k})_{k}$  be a $(C,\e)$-RIS in $X_{W}$ and $(\lambda_{k})_{k}\in
   c_{00}(\N)$ be a sequence  of scalars. Then for every  $f\in W$ of
   type $I$ we can find  $g_{1}$ such that $g_{1}=h_{1}$  or
   $g_{1}=e^{*}_{t}+h_{1}$
   with $t\notin \supp( h_{1})$ where  $h_{1}\in W_{aux}$  with
   $w(h_{1})=w(f)$  and $g_{2}\in c_{00}(\N)$ with $\norm{g_{2}}_{\infty}\leq\e$
   with $g_{1},g_{2}$ having nonnegative coordinates  and such that
   \begin{equation}
     \label{eq:8}
    \abs{f(\sum_{k}\lambda_{k}x_{k})}\leq C(g_{1}+g_{2})(\sum_{k}\abs{\lambda_{k}}e_{k}). 
  \end{equation}
  If we additionally assume that there exists $j_{0}\in\N$ such that for
  every $h\in X_{W}$ with $w(h)=m_{j_{0}}$ and every interval $E$ of
  $\N$ we have that
  \begin{equation}
    \label{eq:9}
    \abs{h(\sum_{k\in E}\lambda_{k}x_{k})}\leq C(\max_{k\in E}\abs{\lambda_{k}}+\e\sum_{k\in E}\abs{\lambda_{k}}),
  \end{equation}
  then if $w(f)\ne m_{j_{0}}$, we may select  $h_{1}$ to have a tree
  analysis $(h_{t})_{t\in\tree[T]}$ with $w(h_{t})\ne m_{j_{0}}$ for all
  $t\in\tree[T]$ with $h_{t}$ a  weighted functional.
 \end{proposition}
We refer  the reader  to  \cite{ATod} Proposition
II.14  or  \cite{AManStudia} Proposition 4.3  for a proof.
 As a first application of the basic inequality 
 we get the following result concerning the action of weighted functionals on RIS.
 \begin{corollary}\label{eris}
   Let $(x_{k})_{k=1}^{n_{j}}$ be  a $(C,m_{j}^{-2})$-RIS and
$x=n_{j}^{-1}\sum_{k=1}^{n_{j}}x_{k}$.  Then 
for
   $f\in W$   with $w(f)=m_{i}$  we have that
   \begin{equation}
     \label{eq:7}
    \abs{f(x)}\leq
    \begin{cases}
      \frac{3C}{m_{i}m_{j}}\,\, &\text{if}\,\, i<j\\
            \frac{C}{m_{i}}+\frac{2C}{n_{j}}\,\, &\text{if}\,\, i\geq j\\
    \end{cases}.
  \end{equation}
  In particular for every $j\in\N$, $\displaystyle m^{-1}_{2j}\leq
  \norm{n_{2j}^{-1}\sum_{k=1}^{n_{2j}}x_{k}}\leq 2Cm_{2j}^{-1}$. 
\end{corollary}

\begin{definition}
 A pair $(x,f)$  with $x\in X_{W}$ and $f\in W$ is said to be a
 $(C,j)$-exact pair if the following conditions are satisfied:
 \begin{enumerate}
 \item $1\leq \norm{x}\leq C$, for every weighted functional $g\in W$  with $w(g)=m_{i}$,
   $i\ne j$ we have $\abs{g(x)}\le \frac{3C}{m_{i}}$ if $i<j$ , while
   $\abs{g(x)}\leq\frac{C}{m_{j}^{2}}$ if $i>j$.
 \item $f$ is  weighted   and $w(f)=m_{j}$.
 \item $f(x)=1$ and $\ran(f)\subset\ran(x)$.
 \end{enumerate}
\end{definition}
The primitive  example of a $(2,2j)$-exact pair is
\begin{equation}
  \label{exactpair}
(\frac{m_{2j}}{n_{2j}}\sum_{k=1}^{n_{2j}}e_{k},m_{2j}^{-1}\sum_{k=1}^{n_{2j}}e_{k}^{*}).  
\end{equation}
Indeed using   Lemma~\ref{basisaverages} and Remark~\ref{basisaverages2}  
it  readily follows that  the pair given  by \eqref{exactpair} is a $(2,2j)$-exact pair.
Also  for  every  block subspace  $Y$ of $X_{W}$     we can find a
$(6,2j)$-exact pair  $(y,f)$ with $y\in Y$.
Indeed let 
$(x_{k})_{k=1}^{n_{2j}}$ be  a $(3,m_{2j}^{-3})$-RIS      where
every $x_{k}\in Y$ is a  $2-\ell_{1}^{n_{2j_{k}}}$-average with $\norm{x_{k}}>1$.
Choose for every $k$  a $f_{k}\in W$, such that $f_{k}(x_{k})=1$ and
$\ran(f_{k})\subset\ran(x_{k})$.
Then from Proposition~\ref{eris}
we get that that
pair  $(y,f)$ where  $y=\frac{m_{2j}}{n_{2j}}\sum_{k=1}^{n_{2j}}x_{k}$  and
$f=m_{2j}^{-1}\sum_{k=1}^{n_{2j}}f_{k}$, is  a  $(6,2j)$-exact pair.

\begin{definition}\label{defSEP}
  Let $C\geq 1$ and $j_{0}\in\N, j_{0}\geq2$.  We call a pair $(x,f)$ for $x\in
  X_{W}$ and $f\in W$ a $(C,j_{0})$-standard exact pair (SEP) if
  there exists a $(C, m_{j_{0}}^{-2})$-RIS  $(x_{k})_{k=1}^{n_{j_{0}}}$ with
  \begin{enumerate}
  \item $x=\frac{m_{j_{0}}}{n_{j_{0}}}\sum_{k=1}^{n_{j_{0}}}x_{k}$.
 \item $f=m_{j_{0}}^{-1}\sum_{k=1}^{n_{j_{0}}}f_{k}$ where $f_{k}\in W$
   with $f_{k}(x_{k})=1$ and $\ran(f_{k})=\ran(x_{k})$.
    \item $x_{k}$ is a  $\ell_{1}^{n_{k}}$-average and  $n_{j_{0}}^{2}<m_{n_{1}}$. 
  \end{enumerate}
\end{definition}
 
\begin{proposition}\label{basicinequality2}
Let $j\in\N$    and $(x_{i},f_{i})_{i=1}^{n_{2j-1}}\subset X_{W}\times W$ such that
$(x_{i})_{i=1}^{n_{2j-1}}$ is a RIS  and  $(x_{i},f_{i})$ is  standard
exact pair for every $i$.   Let  $x=n_{2j-1}^{-1}\sum_{i=1}^{n_{2j-1}}x_{i}\in
X_{W}$ and assume that  for every $g\in W$ with $w(g)=m_{2j-1}$ we have
$\abs{g(x)}\leq \frac{C}{m^{2}_{2j-1}}$.  Then
\begin{equation}
  \label{eq:6}
  \norm{x}_{W}\leq\frac{4C}{m_{2j-1}^{2}}.
  \end{equation}
\end{proposition}
This result follows  from the basic inequality and especially the moreover part. We refer the reader 
to   \cite{AManStudia} 
Proposition 4.9   or \cite{ATod} Proposition II.19  for a proof.

A consequence of the basic inequality is the reflexivity of the
space $X_{W}$.
 \begin{proposition}\label{reflex1}
   The space $X_{W}$ is reflexive.
 \end{proposition}
 We refer to \cite{ATod} Proposition II.28, for a proof.
 As a consequence of the reflexivity of the space $X_{W}$  and the rational convexity of the set $W$ we get the following corollary.
 \begin{corollary}\label{dualball1}
For the space   $X_{W}$ we have that $B_{ X^{*}_{W}}=\overline{W}^{\norm{\cdot}}$.
 \end{corollary}

\section{Standard HI spaces}\label{standardhi}
In this section we show how we can refine    $W$ in order to get
a norming set $W_{hi}$ which will be the norming set of an HI-space.
We start with the definition of the special sequences.

  Let $\Q_{s}=\{(f_{1},f_{2},\dots, f_{l}): f_{i}\in c_{00}(\N),
  f_{i}\ne0, f_{i}(n)\in\Q\, \forall n\in\N, \forall i\leq l, f_{1}<f_{2}<\dots<f_{l}\}$.
We fix a pair  $\Omega_{1},\Omega_{2}$ of disjoint infinite subsets  of $\N$.
  From the fact that  $\Q_{s}$ is countable  there exists an injective
  coding function 
  $\sigma: \Q_{s}\to \{2j: j\in\Omega_{2}\}$   such  that
  \begin{equation}
    \label{eq:16}
        m_{\sigma(f_{1},f_{2}\dots,f_{l} ) }>
     \max\{
     \frac{1}{
       \abs{f_{i}(e_{k})
       }
     }:
     k\in\supp
     f_{i},i\leq l
     \}\cdot
    \maxsupp f_{l}.
    \end{equation}
\begin{definition}\label{sseq}
For  $j\in\N$ and $k\leq n_{2j-1}$  a sequence $(f_{1},\dots,f_{k})$ of
elements  of  $W$ is said to be a $(2j-1)$ special sequence   if
$w(f_{1})=m_{2j_{1}}$ with $j_{1}\in \Omega_{1}$, $m_{2j_{1}}\geq n_{2j-1}^{2}$  and
$w(f_{i+1})=m_{\sigma(f_{1},\dots, f_{i})}$ for every $i\geq 1$.
\end{definition}
\begin{remark}\label{treespecial}
The significant property    of  special sequences    is the tree
like behavior, i.e.,  if $j_{1},j_{2}\in\N$,
$(f_{i})_{i=1}^{n_{2j_{1}}-1}$  is an $(2j_{1}-1)$-special sequence,
$(g_{i})_{i=1}^{n_{2j_{2}}-1}$  is an $(2j_{2}-1)$-special sequence
and if for some $k,l$ we have  $w(f_{k})=w(g_{l})$, then $k=l$  and
$f_{i}=g_{i}$ for every $i< k$.  Moreover if $f_{k}\neq g_{k}$  then
$w(f_{k+l})\ne  w(g_{k+i})$ for every $i,l\geq 1$.
\end{remark}
\subsection{The norming set of a standard HI-space}
We consider the minimal  subset $W_{hi}$  of $c_{00}(\N)$
satisfying the following
\begin{enumerate}
\item $\pm e^{*}_{n}\in W_{hi}$ for every $n\in\N$.
 \item $W_{hi}$  is closed in the $(\mc{A}_{n_{2j}},m_{2j}^{-1})$-operations.
 \item   $W_{hi}$  is closed in the
   $(\mc{A}_{n_{2j-1}},m_{2j-1}^{-1})$-operations   on special sequences, i.e.,
for every 
$(2j-1)$-special sequence  $(f_{1},\dots,f_{k})$ with $f_{i}\in W_{hi}$
for every $i$,   the functional 
$
m_{2j-1}^{-1}\sum_{i=1}^{k}f_{i}\in W_{hi}$.
\item  $W_{hi}$ is symmetric
  and closed under the restriction of its elements to intervals of
  $\N$.
\item $W_{hi}$ is closed under rational convex combinations.
\end{enumerate}
The set  $W_{hi}$ is defined inductively as follows: Set
$W_{hi}^{0}=\{\pm e^{*}_{n}:n\in\N\}$  and assume that $W_{hi}^{n}$ has
been defined. To get $W_{hi}^{n+1}$ , first  we close the set of the
functionals $\cup_{k=1}^{n}W_{hi}^{k}$ in the  $(\mc{A}_{n_{2j}},
m_{2j}^{-1})$-operations,
%i.e. for every $j$ and every $\mc{A}_{2j}$-admissible
%sequence $(f_{1},\dots f_{k})$ of functionals from the set  %$\cup_{k=1}^{n}W_{hi}^{k}$  we take the functional $f=m_{2j}^{-1}}\sum_{i=1}^{k}f_{i}$
and then  we close the above set of functionals in the
$(\mc{A}_{n_{2j-1}},m_{2j-1}^{-1})$ operations on special sequences. The set $W_{hi}^{n+1}$
is the set of the rational convex combinations of restrictions to intervals of the functionals we
defined in the above two steps.

We denote   by  $X_{hi}$ the completion of $c_{00}(\N)$ under the norm induced
by $W_{hi}$.

Note that by the definition  of $W_{hi}$ we get  that it is an  EO-complete
subset of $W_{mT}$ and hence all the results of the previous section
hold for the space $X_{hi}$.

\begin{definition}
 Let  $j\in\N$. A sequence
 $(x_{1},f_{1},x_{2},f_{2},\dots,x_{n_{2j-1}},f_{n_{2j-1}})$ with
 $x_{i}\in X_{hi}$, $f_{i}\in W_{hi}$ is said to be $(C,2j-1)$-dependent   if

 \begin{enumerate}
 
 \item  $(f_{1},\dots,f_{n_{2j-1}})$ is a $W_{hi}-(2j-1)$ special sequence
 
  \item each pair $(x_{i},f_{i})$ is a $(C,2j_{i})$-SEP  where $m_{2j_{i}}=w(f_{i})$.
 
 \end{enumerate}
 
\end{definition}

It follows from the definition of a $(C,2j-1)$-dependent sequence
$(x_{i},f_{i})_{i=1}^{n_{2j-1}}$  that the block sequence
$(x_{i})_{i=1}^{n_{2j-1}}$ is a $3C$-RIS.
\begin{remark}\label{hiproperty}
From the results of the previous
section we get that in every block subspace $Y$ of $X_{hi}$ we can
find for every $j\in\N$  an $(3,2j)$-SEP  $(x,f)$ with  $x\in Y$. Hence  if $Y,Z$ are any two
block subspaces of $X_{hi}$  we can chooce a
a $(3,2j-1)$-dependent sequence $(x_{i},f_{i})_{i=1}^{n_{2j-1}}$
with $x_{2i-1}\in Y$  and $ x_{2i}\in Z$ for every $i$.
\end{remark}
  Regarding the norm of the average
  $n_{2j-1}^{-1}\sum_{i=1}^{n_{2j-1}}x_{i}$
  of the vectors  of a 
  $(C,2j-1)$-dependent sequence we have the following result.

  \begin{proposition}\label{smallnorm}
Let  $(x_{1},f_{1},x_{2},f_{2},\dots,x_{n_{2j-1}},f_{n_{2j-1}})$
be $(C,2j-1)$-dependent sequence.  Then
\begin{equation}
  \label{eq:5}
  m_{2j-1}^{-1}\leq
\norm{\sum_{i=1}^{n_{2j-1}}\frac{x_{i}}{n_{2j-1}}}\,\,\,\andd\,\,
\norm{\sum_{i=1}^{n_{2j-1}}(-1)^{i+1}\frac{x_{i}}{n_{2j-1}}}\leq4 Cm^{-2}_{2j-1}.
\end{equation}

\end{proposition}

 The first inequality in \eqref{eq:5} follows easily. Indeed if $(x_{1},f_{1},x_{2},f_{2},\dots,x_{n_{2j-1}},f_{n_{2j-1}})$
 is a $(C,2j-1)$-dependent  sequence, the properties of the set $W_{hi}$ yield that
 $F=m_{2j-1}^{-1}\sum_{i=1}^{n_{2j-1}}f_{i}\in W_{hi}$. Using this functional
  the first inequality in \eqref{eq:5} readily follows.  Note  also
 $F(\sum_{i=1}^{n_{2j-1}}(-1)^{i+1})x_{i})=0$. It is proved also that
 \[\abs{g(\sum_{i=1}^{n_{2j-1}}(-1)^{i+1}x_{i})}\leq2C(1+ 2m_{2j-1}^{-2})\quad\text{ for every $g\in W_{hi}$ with
     $w(g)=m_{2j-1}$}.\]
 This inequality is a consequence of the following
 two facts:

 The  first one is the tree structure of the special sequences (see
 Remark~\ref{treespecial}). The second one is that  every pair
 $(x_{k},f_{k})$ of a dependent  sequence
 $(x_{i},f_{i})_{i=1}^{n_{2j-1}}$ is a $(C,2j_{k})$-SEP. This
 yields that for every $g\in W_{hi}$ with  $w(g)\ne m_{2j_{k}}$  we have
 $\abs{g(x_{k})}\leq 3C\max\{w(g)^{-1}, m_{2j_{k}}^{-2}\}$
 Combining  the above two facts we get that  the assumption  of
 Proposition~\ref{basicinequality2} holds. This gives us
 the second inequality  in \eqref{eq:5}.

As a corollary of the previous proposition we get the following
corollary.
\begin{corollary}
 The space $X_{hi}$  is an HI space.
\end{corollary}
Indeed let $Y,Z$ be any two block subspaces  of $X_{0}$ and $j\in\N$.
From  Remark~\ref{hiproperty}  we  get
a $(3,2j-1)$-dependent sequence $(x_{i},f_{i})_{i=1}^{n_{2j-1}}$
with $x_{2i-1}\in Y$  and $ x_{2i}\in Z$ for every $i$.
From  \eqref{eq:5}    we  get  that
 \begin{equation}
   \label{eq:4}
  \norm{\sum_{i=1}^{n_{2j-1}}(-1)^{i+1}x_{i}}\leq \frac{12}{m_{2j-1}}\norm{\sum_{i=1}^{n_{2j-1}}x_{i}}.
\end{equation} 
Since $j\in\N$ the above inequality yields that $\dist(S_{Y},S_{Z})=0$ and this implies
that   $X_{0}$ is an HI space.

%\newpage

\section{The defintion of the space $X_{ex}$.}\label{swex}
In this section we define the HI space $X_{ex}$ which contains as an isometric subspace the initial HI space $X_{hi}$.
For this we will define  a norming set $W_{ex}$, a subset of  $c_{00}(\N)$, that  induces  the norm of the space. Actually the set  $W_{ex}$  is a
subset  of the set $W_{mT}$ which is the norming set of the space
$X_{mT}=T[(\mc{A}_{n_{2j}},m_{2j}^{-1})\cup(\mc{A}_{2n_{2j-1}},m_{2j-1}^{-1})]$.
Note that $X_{mT}$  uses in the odd operations the families $\mc{A}_{2n_{2j-1}}$  instead of $\mc{A}_{n_{2j-1}}$ which  were used in the definition of $X_{hi}$. 
\subsection{The definition of  $W_{ex}$.}\label{EOcomplete}
Let $\N=\N_{od}\cup\N_{ev}$, $\N_{od}$ the odd and $\N_{ev}$  the even
positive integers.

We shall denote also by $K=\cup_{j}K^{2j}$ the subset of $W_{mT}$, where
\[
K^{2j}=\{m_{2j}^{-1}\sum_{i\in F}e_{i}^{*}: \#F=n_{2j}\}.
\]
For every $j$, let   $(K_{k}^{2j})_{k\in\N}$,
be disjoint subsets of $K^{2j}$  such that  $\sup\{\min(f):f\in
K_{k}^{2j}\}=+\infty$ for every $k$.
For $g=\sum_{i=1}^{+\infty}a_{i}e^{*}_{i}\in c_{00}(\N)$ we denote by $\widehat{g}$  the
vector $\widehat{g}=\sum_{i}a_{i}e_{2i}^{*} \in c_{00}(\N)$.
We shall denote by $\widehat{K_{k}^{2j}}$  the set $\{\widehat{f}: f\in
K^{2j}_{k}\}$ and by $\widehat{W}_{hi}$  the
set  $\{\widehat{f}:f\in W_{hi}\}$.
\begin{definition}[The coding functions $\varrho_{k}^{2i}, k\in\N, i=0,1,\dots$]\label{varrho}
Let   $(D_{k})_{k\in\N}$  be a sequence of disjoint  subsets of $\wmt$, to be determined later.
For every $i,k\in\N$ we shall denote by $L_{k}^{2i}$ the following subset of $D_{k}$,
\[L_{k}^{2i}=\{f\in D_{k}:w(f)=m_{2i}\,\,\andd \supp(f)\subset\N_{od}\}.\]
For every $k,i\in\N$  let $\varrho_{k}^{2i}:L_{k}^{2i}\to K_{k}^{2i}$
 be an injection such that
\begin{align}
\text{for every $f\in L_{k}^{2i}$ we have}\quad f<\widehat{\varrho_{k}^{2i}(f)}.\quad\label{p1}
\end{align}
\end{definition}
\begin{remark}\label{r5.1}
  From the above definitions it  follows that if $g\in K_{k}^{2i}$  for some $i,k\in\N$  then there is no $n\ne k$ and $f\in L_{n}^{2i}$ such that $g=\varrho_{n}^{2i}(f)\in K_{n}^{2i}$.

  This follows from the fact that the domains of $\varrho_{k}^{2i},\varrho_{n}^{2i}$  are disjoint as well as their ranges 
$K_{k}^{2i},K_{n}^{2i}$.
\end{remark}

Our goal in this subsection is to define an  EO-complete subset   $W_{ex}$
of $W_{mT}$ such that  the restriction of every element $f\in W_{ex}$  to the
set $\N_{ev}$ belongs to  $\widehat{W}_{hi}$ and in particular
$\widehat{W}_{hi}=\{f_{|\N_{ev}}: f\in W_{ex}\}$.
Denoting by $X_{ex}$  the space defined by the set $W_{ex}$, the
aforementioned  property yields that  its subspace
$\widehat{X}_{hi}$ generated by the subsequence
$(e_{2n})_{n\in \N}$ of the basis  is isometric to the initial
HI space $X_{hi}$.    We shall also include in $W_{ex}$ some new special
functionals, named paired special functionals  which will permit us
to export the  HI structure of the space  $\widehat{X}_{hi}$ to the
whole space $X_{ex}$.

We proceed to give three definitions of special functionals which will be used in the definition of the norming    $W_{ex}$. We assume the existence of a family  $(D_{k})_{k\in\N}$ of disjoint subsets of $W_{mT}$ that will be specified later and we consider the corresponding
 functions $\varrho_{k}^{2i}:L_{k}^{2i}\to K_{k}^{2i}$, $k\in\N$,$i\geq 0$ as they were defined in Definition~\ref{varrho}

\begin{definition}\label{pairedspecials}
  Let $j,n\in\N$ and $V\subset W_{mT}$.
We call a 
$(2j-1,V,(\varrho_{k}^{2i})_{k=1}^{n}, i\geq0)$- paired special
sequence
any block sequence
$(f_{1},\widehat{g}_{1},\dots, f_{p},\widehat{g}_{p})$ satisfying
the following:
\begin{enumerate}
\item  $p\leq n_{2j-1}$ and for every $k\leq p$ there exist  $j_{k}\in\N$   and $i_{k}\leq n$  such   that    $f_{k}\in V\cap L_{i_{k}}^{2j_{k}}$ and  $g_{k}=\varrho^{{2j_{k}}}_{i_{k}}(f_{k})$.
\item The finite sequence    $(g_{1},g_{2},\dots,g_{p})$ is a $(2j-1)$-special
    sequence of $W_{hi}$.
\end{enumerate}
\end{definition}

\begin{definition}\label{semispecials}
  Let $j,n\in\N$ and $V\subset W_{mT}$.
We call a $(2j-1,V,(\varrho_{k}^{2i})_{k=1}^{n}),i\geq0)$-semi paired
special sequence any
block   sequence   of the  form $(f_{i},\widehat{g}_{i})_{i=1}^{l}\,^{\frown} (\widehat{g}_{i})_{i=l+1}^{p}$
satisfying the following:
\begin{enumerate}
\item  $p\leq n_{2j-1}$, $l\geq 0$   and
  $(f_{1},\widehat{g}_{1}f_{2},\widehat{g}_{2},\dots,f_{l},\widehat{g}_{l})$ is a
$(2j-1, V,(\varrho_{i}^{2k})_{i=1}^{n},k=0,1,\dots)$-paired special sequence.
\item  For every $m=l+1,\dots,p$ we have that $g_{m}\in
  K^{2j_{m}}_{i_{m}}$ for some  $i_{m}\leq n$,
and  $(g_{1},g_{2},\dots,g_{p})$ is a $(2j-1)$-
    special sequence  of $W_{hi}^{1}$.
\item  There is no   $\phi\in V$ such that
  $(f_{1},\widehat{g}_{1}f_{2},\widehat{g}_{2}\dots,f_{l},\widehat{g}_{l},\phi,\widehat{g}_{l+1} )$   is a
  $(2j-1, V,(\varrho_{k}^{2i})_{k=1}^{n},i=0,1,\dots)$ paired special sequence. 
\end{enumerate}
\end{definition}

The  third definition is the following.
\begin{definition}\label{wspecials}
  Let $j,n\in\N$ and $V\subset \wmt$. We  call a  $(2j-1,V,(\varrho_{k}^{2i})_{k=1}^{n},i\geq 0)$-special sequence any
block  sequence  
  $(f_{1},\widehat{g}_{1},\dots,f_{l},\widehat{g}_{l},\widehat{g}_{l+1},\dots
  ,\widehat{g}_{m},h_{m+1},h_{m+2},\dots,h_{p})$
  satisfying the following:
  \begin{enumerate}
  \item  $p\leq n_{2j-1},$ $0\leq l\leq m$  and
 $(f_{1},\widehat{g}_{1},\dots,f_{l},\widehat{g}_{l},\widehat{g}_{l+1},\dots
  ,\widehat{g}_{m})$ is a  $(2j-1, V,
    (\varrho_{k}^{2i})_{k=1}^{n},i=0,1,\dots)$-semi paired special sequence.
  \item $h_{m+1|\N_{ev}}\in  V\setminus \cup_{j=0}^{+\infty}\cup_{i=1}^{n}\widehat{K}_{i}^{2j}$ and
for every $q\geq1$,    $h_{m+q|\N_{ev}}=\widehat{\phi}_{m+q}\in\widehat{W}_{hi}$ and
$(g_{1},\dots,g_{m},\phi_{m+1},\phi_{m+2},\dots,\phi_{p})$ is a
$(2j-1)$-special sequence of $W_{hi}$.
  \end{enumerate}
  
\end{definition}

\begin{remark}\label{treestructure}
 It follows from the above definitions that   the
 paired special
sequences  are semi paired special seqeunces
and the later are special sequences.
\end{remark}
\subsection{Definition of the norming set $W_{ex}$}
We proceed by induction to define an increasing family of sets $W_{n}$, $n=0,1,2,\dots$ which will be subsets of $W_{mT}$  such that   $W_{ex}=\cup_{n=0}^{\infty}W_{n}$.  For every $n\geq 1$ we shall choose  a set  $D_{n}\subset W_{n}$ such that the sets $(D_{n})_{n\in\N}$ are disjoint and for every $n\in\N$, $i\geq 0$, we will consider the injective function $\varrho_{n}^{2i}:L_{n}^{2i}\to K_{n}^{2i}$ as it was defined in Definition~\ref{varrho}.

Let  $W_{0}=\{\pm e^{*}_{n}:n\in\N\}$ .
Let $n\in\N$ and assume that for every  $k<n$ the sets  $D_{k}\subset W_{k}$ and the functions $\varrho_{k}^{2i}$, $i\geq 0$ have been defined.

We set 
 \[
 U_{n}=\cup_{j}\{m_{2j}^{-1}\sum_{i=1}^{k}f_{i}: (f_{i})_{i=1}^{k}\,\,
 \text{is an $\mc{A}_{n_{2j}}$-admissible  sequence of}\,\,  W _{n-1}\},
 \]
 and $V_{n}=(U_{n}\setminus\cup_{j=0}^{\infty}\cup_{i=n+1}^{+\infty}\widehat{K_{i}^{2j}})\cup W_{n-1}$.

 We set  $D_{n}=V_{n}\setminus W_{n-1}$ and for  $i\geq 0$ we define 
\[L_{n}^{2i}=\{f\in D_{n}: w(f)=m_{2i}\,\,\text{and}\,\,\supp(f)\subset \N_{od}\}.
\]
 For every $i\in\N$ we a choose an injective function 
$\varrho_{n}^{2i}:L_{n}^{2i}\to K_{n}^{2i}
$
satisfying  $f<\widehat{\varrho_{n}^{2i}(f)}$ for every $f\in L_{n}^{2i}$.

It is clear that $D_{n}\cap W_{n-1}=\emptyset$
and hence  $D_{n}$ is disjoint from every $W_{k}$, $k<n$.   Also we have that $L_{n}^{2i}\subset V_{n}$ for every $i\geq 0$.

We define now  for every $j\in\N$ the
$(2j-1,V_{n},(\varrho_{k}^{2i})_{k=1}^{n},i\geq0)$
special functionals.

For $j\in\N$ we  define
\begin{align}
  W_{n}^{2j-1}&=\{\frac{1}{m_{2j-1}}\left(\sum_{i=1}^{l}(f_{i}+\widehat{g}_{i})+
    \sum_{i=l+1}^{m}\widehat{g}_{i}+\sum_{i=m+1}^{p}h_{i}\right):
 \label{eq3}
  \\
 &
 \resizebox{0.92\hsize}{!}{$
  (f_{1},\widehat{g}_{1},\dots,f_{l},\widehat{g}_{l},
  \widehat{g}_{l+1},\dots\widehat{g}_{m},h_{m+1},\dots,h_{p})\,
\text{is a}\,\,
   (2j-1, V_{n},(\varrho^{2i}_{k})_{k=1}^{n}, i=0,1,\dots)$}
\notag
  \\
&\hspace{7cm}\text{special sequence}.
\}
                                                   \notag
\end{align}
We set
\[W_{n}=\conv_{\Q}\{Ef:  f\in  V_{n}\cup\cup_{j=1}^{\infty}
  W_{n}^{2j-1}:E\,\text{interval of}\,  \N\}.\]

The norming set $W_{ex}$ is the set $W_{ex}=\cup_{n=0}^{+\infty}W_{n}$.

From the inductive   assumption  we have that  $W_{n-1}$ is a subset  of  set $W_{mT}$. It follows that the set $U_{n}$   is also subset of $W_{mT}$.  Hence for every $j\in\N$ the set   $W_{n}^{2j-1}$ is a subset of $W_{mT}$ and finally the set $W_{n}$  is subset of $W_{mT}$.
 It is clear also that the sets $W_{n-1}$ and $D_{n}$ are subsets of $W_{n}$.
\begin{remarks}~\label{remarksonspecial}
  \begin{enumerate}
\item[a)]If $(f_{1},\widehat{g}_{1},f_{2},\widehat{g}_{2}, f_{l},\widehat{g}_{l}, \widehat{g}_{l+1},\widehat{g}_{l+2},\dots,\widehat{g}_{p})$
is  a 
$(2j-1, V_{n},(\varrho^{2i}_{k})_{k=1}^{n},i=0,1,\dots)$-semi-paired
special and $l\geq 1$,
then for the functional $\widehat{g}_{l+1}$  we have that 
\begin{enumerate}
\item[i)]  either  $g_{l+1}\notin\cup_{k=1}^{n}\varrho_{k}^{2j_{l+1}}(L_{k}^{2j_{l+1}})$, $w(g_{l+1})=m_{2j_{l+1}}$,
\item[ii)] or   $g_{l+1}=\widehat{\varrho^{2j_{l+1}}_{k_{l+1}}(f_{l+1})}$ for  an  $f_{l+1}\in L_{k_{l+1}}^{2j_{l+1}}$, $k_{l+1}\leq n$
  and    $\minsupp f_{l+1}\leq\maxsupp \widehat{g}_{l}$.
  
Note that in both  cases  the definition of the  coding functions $\varrho_{k}^{2i}$ and the disjointness of the sets $K_{k}^{2i}$  implies that
there  is no  $m>n$  and a $(2j-1,V_{m},(\varrho_{k}^{2i})_{k=0}^{m},i\geq0)$- paired
special sequence  having as an initial segment the
sequence
$(f_{i},\widehat{g}_{i})_{i=1}^{l+1}$.
\end{enumerate}
\item[b)] For every $g_{1}\in \varrho^{2j_{1}}_{i_{1}}(L_{i_{1}}^{2j_{1}})$  with
$w(g_{1})=m_{2j_{1}}\geq n^{2}_{2j-1}$ for some $j\in\N$ and $i_{1}\leq n$,
we have that
$((\varrho^{2j_{1}}_{i_{1}})^{-1}(g_{1}), \widehat{g}_{1})$ is a  $(2j-1, V_{n},(\varrho^{2i}_{k})_{k=1}^{n},i=0,1,\dots)$-paired sequence.
\item[c)]
If $(g_{1},\dots,g_{p})$ is a  $(2j-1)$-special sequence in $W_{hi}$,
such that for every $l\leq p$,
$g_{l}\in K^{2j_{l}}$ for some $4j_{l}\in\N$  then if $\widehat{g}_{1}\notin \cup_{k=1}^{\infty}K_{k}^{2j_1}$  it follows that
$(\widehat{g}_{1},\dots,\widehat{g}_{p})$ is a $(2j-1, V_{n},(\varrho^{2i}_{k})_{k=1}^{n},i=0,1,\dots)$  special sequence where
$n=\max\{1, k:\exists l\leq p\,\,\text{with}\, g_{l}\in K_{k}^{2j_{l}}
\}$.

If $g_{1}\in K_{k_{1}}^{2j_{1}}\setminus\varrho_{k_{1}}^{2j_{1}}(L_{k_{1}}^{2j_{1}})$  then
$(\widehat{g}_{1},\dots,\widehat{g}_{p})$ is a $(2j-1, V_{n},(\varrho^{2i}_{k})_{k=1}^{n},i=0,1,\dots)$-semi paired special sequence where
$n=\max\{k:\exists l\leq p\,\,\text{with}\, g_{l}\in K_{k}^{2j_{l}}
\}$.

% (\varrho^{2i}_{k})_{k=1}^{n},i=0,1,\dots)$ semi-paired special sequence.

% $g_{i}\in K^{2j_{i}}_{l_{i}}$    and  $n=\max\{l_{i}:i\leq p\}$
% then if $g_{1}\notin \varrho^{2j_{1}}_{k}(L_{l_1}^{2j_{1}})$ it follows 
% that $(\widehat{g}_{1},\dots,\widehat{g}_{p})$ is a $(2j-1, V_{n},
% (\varrho^{2i}_{k})_{k=1}^{n},i=0,1,\dots)$ semi-paired special sequence.
\end{enumerate}
\end{remarks}

 It follows form the definition of $W_{ex}$ that it is an EO-complete
 subset of  $W_{mt}$. Hence from the results  of
 Section~\ref{eocomplete}
 we have the following.
 \begin{proposition}
The space $X_{ex}$ is reflexive.  Moreover  $B_{X_{ex}^{*}}=\overline{W_{ex}}^{\norm{}}$.   
\end{proposition}

We show now  that  the semi paired special sequences have the tree property, a result that will be essential in the next sections.

\begin{notation}
If  $(f_{1},\widehat{g}_{1},f_{2},\widehat{g}_{2}, f_{l},\widehat{g}_{l},\widehat{g}_{l+1},\dots, \widehat{g}_{m})$
  is a
  $(2j-1, V_{q},(\varrho^{2i}_{k})_{k=1}^{q},i=0,1,\dots)$-semi paired
  special sequence, the pair defined by $g_{i}$, $i\leq m$ is defined to be the pair  $\{f_{i},g_{i}\}$ if $i\leq l$  and the singleton $\{g_{i}\}$ if $i=l+1,\dots,m$. The weight of the pair is defined to be the weight of $g_{i}$.
\end{notation}

\begin{lemma}
 The set of the semi paired special sequences has a tree structure, i.e., if
$(f_{i}^{1},\hat{g_{i}^{1}})_{i=1}^{p_{1}}{}^{\frown}(\hat{g_{i}^{1}})_{i=p_{1}+1}^{m_{1}}$
 and
 $(f_{i}^{2},\hat{g_{i}^{2}})_{i=1}^{p_{2}}{}^{\frown}(\hat{g_{i}^{2}})_{i=p_{2}+1}^{m_{2}}$
  are two semi-paired  special sequences and
the couples determined by some $g_{i}^{1},g_{j}^{2}$ have the same weight
  then  $i=j$  and  then couples determined by $g_{l}^{1}$ and $g_{l}^{2}$ are equal for all $l\leq i$.
 Moreover if $g_{i}^{1}\ne g_{i}^{2}$  then all the couples determined by
 $g_{k}^{1},g_{p}^{2}$, $k,p>i$ have different weights.
\end{lemma}
\begin{proof}
  By the definition of the semi-paired  special sequences it follows that
  $(g_{l}^{1})_{l=1}^{i}$ and $(g_{k}^{2})_{k=1}^{j}$ are special sequences of $W_{hi}^{1}$. Since the special sequences of $W_{hi}$ have the tree property and $w(g_{i}^{1})=w(g_{j}^{2})$ 
  it follows that  $i=j$   and   $g_{k}^{1}=g_{k}^{2}$ for all $k< i=j$.
  Also from the definition of the (semi)-paired special sequences and the  injectivity of the funtions  $\varrho_{n}^{2j}$,$n\in\N,j\geq0$ it follows that
 the couples determined by $g_{k}^{1},g_{k}^{2}$, $k<i$, are equal and in particular they have the same cardinality.

  The same holds of $g_{i}^{1}=g_{i}^{2}$.    In the case $g_{i}^{1}\ne g_{i}^{2}$ again from the tree property of the special sequences of $W_{hi}$ it follows that
  $w(g_{k}^{1})\ne w(g_{l}^{2})$ for all $k,l>i$.
% \end{proof}
% \begin{remark}
% If
%  $(f_{i}^{1}\hat{g_{i}^{1}})_{i=1}^{p_{1}}\frown(\hat{g_{i}^{1}})_{i=p_{1}+1}^{m_{1}}$
%  and
%  $(f_{i}^{2}\hat{g_{i}^{2}})_{i=1}^{p_{2}}\frown(\hat{g_{i}^{2}})_{i=p_{2}+1}^{m_{2}}$
%   are two semi-paired  special sequences and
%   the couples determined by some $g_{i}^{1},g_{j}^{2}$have the same weight then $i=j$.
% It  may happen that  $i=j=\min\{l_{1},l_{2}\}$
%   and 
%    $(f_{k}^{1}\hat{g_{k}^{1}})_{k=1}^{l_{1}}$ be a paired special sequence and
%    $(f_{k}^{2}\hat{g_{k}^{2}})_{k=1}^{l_{1}-1}\,^{\frown}(g_{l_{1}}^{2})$ be a semi-paired special sequence.

   In the case    $i=j=l_{1}$  (or $i=j=l_{2}$)   we have that    $l_{1}-1\leq l_{2}$ ( $l_{2}-1\leq l_{1}$).

%   In the case  $i=j<\min\{l_{1},l_{2}\}$ or  $i=j>\max\{l_{1},l_{2}\}$ the either both sequences are paired special sequences   or both are semi paired special sequences.
% \end{remark}
\end{proof}

We proceed  to show that $\widehat{W}_{hi}=\{f_{|\N_{ev}}: f\in W_{ex}\}$.
First we show that $W_{ex|\N_{ev}}$ is a subset of $\widehat{W}_{hi}$. This
will be an immediate consequence of the following lemma.
\begin{lemma}\label{step1}
 For every  $n\in\N$ and every $f\in W_{n}$  we have that  $f_{|\N_{ev}}\in \widehat{W}_{hi}$. 
\end{lemma}
\begin{proof}
 We prove the lemma by induction. For  every $n\in\N$, $e^{*}_{n|\N_{ev}}=0$  if
 $n$ is odd otherwise  if $n=2k$, $e^{*}_{n|\N_{ev}}=e^{*}_{n}=\widehat{e^{*}_{k}}\in\widehat{W}_{hi}$.
 Assume that  the result holds for all  $k\leq n$  and $f\in W_{n+1}$.

  Let   $f$ be  the result of an  $(\mc{A}_{n_{2j}}, m_{2j}^{-1})$-
  operation of elements  of $W_{n-1}$, i.e., $f=m_{2j}^{-1}\sum_{i=1}^{d}f_{i}$. Then  from the inductive
  hypothesis we have that $f_{i|\N_{ev}}=\widehat{\phi_{i}}\in \widehat{W}_{hi}$ for every $i$.  Since
  $W_{hi}$ is closed  in the
  $(\mc{A}_{n_{2j}},m_{2j}^{-1})$-operations it  follows that
  $\phi=m_{2j}^{-1}\sum_{i=1}^{d}\phi_{i}\in W_{hi}$
and    $f_{|\N_{ev}}=\widehat{\phi}\in \widehat{W}_{hi}$.

    Assume  now that  there exists a $(2j-1, V_{n},
    (\varrho_{k}^{2i})_{k=1}^{n},i=0,1,\dots)$-special sequence
    $(f_{1},\widehat{g}_{1},f_{l},\widehat{g}_{l},\widehat{g}_{l+1},\dots,
    \widehat{g}_{m},h_{m+1},\dots,h_{p})$  such that
    $f=m_{2j-1}^{-1}(\sum_{i=1}^{l}(f_{i}+\widehat{g}_{i})+\sum_{i=l+1}^{m}\widehat{g}_{i}+\sum_{i=m+1}^{p}h_{i})$.
By the definition of the special sequences we have that
$f_{i|\N_{ev}}=0$ for every $i\leq l$ and $h_{i|\N_{ev}}=\widehat{\phi}_{i}\in
\widehat{W}_{hi}$ for every $i\geq m+1$.
Moreover  $(g_{1},\dots,g_{m},\phi_{m+1},\dotsm\phi_{p})$ is a
$(2j-1)$-special sequence of  $W_{hi}$.  It follows that
\[
  \phi=m_{2j-1}^{-1}(\sum_{i=1}^{m}g_{i}+\sum_{i=m+1}^{p}\phi_{i})\in
  W_{hi}\,\,\,
  and\,\, f_{|\N_{ev}}=m_{2j-1}^{-1}(\sum_{i=1}^{m}\widehat{g}_{i}+\sum_{i=m+1}^{p}h_{i|\N_{ev}})=\widehat{\phi}.\]
%\in \widehat{W}_{hi}.
%  f_{|\N_{ev}}=m_{2j-1}^{-1}\left(\sum_{i=1}^{m}\widehat{g}_{i}+\sum_{i=m+1}^{p}h_{i|\N_{ev}}\right)
%   =
% m_{2j-1}^{-1}\left(\sum_{i=1}^{m}\widehat{g}_{i}+\sum_{i=m+1}^{p}\widehat{\phi_{i}}\right)\in \widehat{W}_{hi}.
%\]
The case   $f$ is  rational convex combination  of functionals of
the first two cases  is immediate consequence  of the fact that
$W_{hi}$ is closed in rational convex combinations. This completes  the proof.
\end{proof}
The next lemma  will be used to prove the converse inclusive relation.
\begin{lemma}\label{wh1specials}
  Let  $g_{i}\in K_{n_{i}}^{2j_{i}}$, $i\leq m$  be such that
  $(g_{1},\dots,g_{m})$ is  a $(2j-1)$-special sequence of
  $W_{hi}$.  Then there exist $0\leq l\leq m$ and $f_{1},\dots,f_{l}\in
  W_{ex}$ such that $(f_{1},\widehat{g}_{1},\dots,f_{l},\widehat{g}_{l},\widehat{g}_{l+1},\dots,\widehat{g}_{m})$
  is a $(2j-1, V_{n}, (\varrho_{k}^{n})_{k=1}^{n},i=0,\dots)$ semi-paired
  special sequence where  $n=\max\{n_{i}:i\leq m\}$.
\end{lemma}
\begin{proof}
Let $l\le m$  be maximal with the properties
\begin{enumerate}
\item $g_{i}\in \varrho_{n_{i}}^{2j_{i}}(L_{n_{i}}^{2j_{i}})$ for every $i\leq l$.
\item
$F=((\varrho^{2j_{i}}_{n_{i}})^{-1}(g_{i}), \widehat{g}_{i})_{i=1}^{l}$  is a  $(2j-1, V_{n},
(\varrho_{k}^{2i})_{k=1}^{n},i=0,1\dots)$ -paired special sequence, where
$n=\max\{n_{i}:i\leq m\}$.
\end{enumerate}

If  $l=0$,   Remark~\ref{remarksonspecial}b) yields  that $g_{1}\notin \varrho_{n_{1}}^{2j_{1}}( L_{n_{1}}^{2j_{1}})$
and   from Remark~\ref{remarksonspecial}c) we get that $(\widehat{g}_{1},\widehat{g}_{2},\dots,\widehat{g}_{m})$ is  
$(2j-1, V_{n},(\varrho_{k}^{2i})_{k=1}^{n},i=0,1,\dots,)$-semi-paired special sequence.

 If $l=m$  we have that 
$ ((\varrho^{2j_{i}}_{n_{i}})^{-1}(g_{i}), \widehat{g}_{i})_{i=1}^{m}$  is   $(2j-1, V_{n},
(\varrho_{k}^{2i})_{k=1}^{n},i=0,1\dots)$ paired   special sequence.

Let $l<m$. Then  for   the functional $g_{l+1}$  we have that either
does not  belong to $\varrho_{n_{l+1}}^{2j_{l+1}}(L_{n_{l+1}}^{2j_{l+1}})$  or belongs  to
$\varrho_{n_{l+1}}^{2j_{l+1}}(L_{n_{l+1}}^{2j_{l+1}})$  
and $\min (\varrho_{n_{l+1}}^{2i_{l+1}})^{-1}(g_{l+1})\leq\max \widehat{g}_{l}$. In both
cases  we get 
that
\[((\varrho^{2j_{1}}_{1})^{-1}(g_{1}), \widehat{g}_{1},\dots,
(\varrho^{2j_{1}}_{l})^{-1}(g_{l}),\widehat{g}_{l},
\widehat{g}_{l+1},\dots,\widehat{g}_{m})\] is
 $(2j-1, V_{n},(\varrho_{k}^{2i})_{k=1}^{n},i=0,1\dots)$ semi-paired special
 sequence.
\end{proof}

\begin{lemma}\label{step2}
For every $g\in W_{hi}$   there exists  $\phi\in W_{ex}$ such that  $\phi_{|\N_{ev}}=\widehat{g}$.
\end{lemma}
\begin{proof}
  We shall prove by induction on $n$  that for every $g\in W_{hi}^{n}$
  there exists  $\phi\in W_{ex}^{n_{0}}$,$n_{0}=n_{0}(g)\geq n$, such that  $\phi_{|\N_{ev}}=\widehat{g}$.
  
  The  case  $g=\pm e^{*}_{k}\in W_{hi}^{0}$ it trivial.   Let $g\in W_{hi}^{1}$.  We shall consider the following  two cases.

  \textit{Case  1}. $g=m^{-1}_{2j}\sum_{i\in F}\pm e_{i}^{*}$ with $\# F\leq n_{2j}$.

We have to consider the following two subcases.

\medskip

\textit{Subcase a)}. $g\notin \cup_{n=2}^{\infty}K_{n}^{2j}$.

In this subcase we get that $\phi=m_{2j}^{-1}\sum_{i\in F}\pm e^{*}_{2i}\in  V_{1}$
and $\phi=\phi_{|\N_{ev}}=\widehat{g}$.

\medskip
\textit{Subcase b)}. $g\in K_{n}^{2j}$ for some  $n\geq 2$.

In this subcase we get that $\phi=m_{2j}^{-1}\sum_{i\in F}e^{*}_{2i}\in  V_{n}\subset W_{n}$
and $\phi=\widehat{g}$.

We pass now to the next case.

\smallskip

\textit{Case 2}.
Let $g=m_{2j-1}^{-1}\sum_{i=1}^{p}g_{i}$  where 
 $g_{i}\in W_{hi}^{1}$ and $(g_{1},\dots,g_{p})$ is a
$(2j-1)$-special sequence.

Let $m\leq p$ be maximal with the property  $g_{k}\in K_{i_{k}}^{2j_{k}}$  for
every  $k\leq m$.
If $m=0$  from ~\ref{remarksonspecial}c) we get that
$(\widehat{g}_{1},\dots,\widehat{g_{m}})$
is  $(2j-1, V_{n_{0}}, (\varrho_{k}^{2i})_{k=1}^{n_{0}},i=0,\dots)$ 
special sequence  for some $n$.

If  $m\geq 1,$ Lemma~\ref{wh1specials}  yields that
there exist $l\leq m$,  $f_{1},\dots, f_{l}\in W_{ex}$  and $n_{0}\in\N$  such
that
$(f_{1},\widehat{g}_{1},\dots,f_{l},\widehat{g}_{l},\widehat{g}_{l+1},\dots,\widehat{g_{m}})$
is a $(2j-1, V_{n_{0}}, (\varrho_{k}^{2i})_{k=1}^{n_{0}},i=0,\dots)$ semi-paired
special sequence.
Clearly by the definition of the
$(2j-1, V_{n}, (\varrho_{k}^{2i})_{k=1}^{n},i=0,\dots)$ special
sequences,
we get that
$(f_{1},\widehat{g}_{1},\dots,f_{l},\widehat{g}_{l},\widehat{g}_{l+1},\dots,\widehat{g_{p}})$
is  $(2j-1, V_{n}, (\varrho_{k}^{2i})_{k=1}^{n},i=0,\dots)$ special
sequence
where  $n=\max\{r_{i}: g_{i}\in K_{r_{i}}^{2j_{i}}, i\in \{1,\dots,p\}\}$.
  It follows
  \[\phi=m_{2j-1}^{-1}\left(\sum_{i=1}^{l}(f_{i}+\widehat{g}_{i})+\sum_{i=l+1}^{p}\widehat{g}_{i}
      \right)\in W_{ex}
\,\,\,\,\text{and}\,\,\,\,
\phi_{|\N_{ev}}=m_{2j-1}^{-1}\sum_{i=1}^{d}\widehat{g}_{i}=\widehat{g}.
\]
The case $g$ is rational convex combination of  functional of the
above two cases  follows readily. This
complete the first inductive step.

Assume that we have prove the result for every $k\leq n$ and let $g\in
W_{hi}^{n+1}$.

The case $g$  is that result of an
$(\mc{A}_{n_{2j}},m_{2j}^{-1})$-operation  of  elements of $W_{n}^{hi}$
follows  easily from the inductive hypothesis.  Indeed if
$g=m_{2j}^{-1}\sum_{i=1}^{d}g_{i}$  then for every $i$ we have that there
exists  $\phi_{i}\in W_{ex}$ such that  $\phi_{i|\N_{ev}}=\widehat{g}_{i}$.
Since $W_{ex}$ is closed in projections on intervals
we may assume that  $\ran(\phi_{i})=\ran(\widehat{g_{i}})$.
Also using that 
$W_{ex}$ is closed in  $(\mc{A}_{n_{2j}},m_{2j}^{-1})$-operations we get
\[
\phi=m_{2j}^{-1}\sum_{i=1}^{d}\phi_{i}\in W_{ex}\,\,\andd  \phi_{|\N_{ev}}=\widehat{g}.
\]
Assume now that  $g=m_{2j-1}^{-1}\sum_{i=1}^{p}g_{i}$  where
$(g_{1},\dots,g_{p})$ is a $(2j-1)$-special sequence of $W_{hi}^{n+1}$

Let $m\leq p$ be maximal with respect the property
$g_{k}\in K_{i_{k}}^{2j_{k}}$ for some $i_{k}$,  for every $k\leq m$.
Then  from  Lemma~\ref{wh1specials}  we get that there exist $l\leq m$,   $f_{1},\dots, f_{l}\in W_{ex}$  and $n_{0}\in\N$  such
that
$(f_{1},\widehat{g}_{1},\dots,f_{l},\widehat{g}_{l},\widehat{g}_{l+1},\dots,\widehat{g_{m}})$
is a $(2j-1, V_{n_{0}}, (\varrho_{k}^{2i})_{k=1}^{n_{0}},i=0,\dots)$  semi-paired
special sequence.
For every $i>m$  from the inductive hypothesis
we have that there exists  $h_{i}\in W_{ex}^{n_{i}}$ such that
$h_{i|\N_{ev}}=\widehat{g_{i}}$.  Since $W_{ex}$ is closed in projections on
intervals we may assume that  $\ran(h_{i})=\ran(\widehat{g_{i}})$ for
every $i$.  Moreover from the choice of $m$ we have also that
$h_{m+1|\N_{ev}}=\widehat{g}_{m+1}\notin \cup_{k=1}^{\infty}\widehat{K}_{k}^{2j_{m+1}}$.

It follows that $
(f_{1},\widehat{g}_{1},\dots,f_{l},\widehat{g}_{l},\widehat{g}_{l+1},\dots,\widehat{g_{m}},
h_{m+1},\dotsm h_{p})$
is a $(2j-1, V_{\bar{n}_{0}}, (\varrho_{k}^{2i})_{k=1}^{\bar{n}_{0}},i=0,\dots)$
special sequence for some  $\bar{n}_{0}\geq n$.  Hence\[
\phi=m_{2j-1}^{-1}\left(\sum_{i=1}^{l}(f_{i}+\widehat{g}_{i})+\sum_{i=l+1}^{m}\widehat{g}_{i}+\sum_{i=m+1}^{p}h_{i}\right)
\in W_{ex}
\andd \phi_{|\N_{ev}}=\widehat{g}.
\]
\end{proof}
As a corollary we get  the following.
\begin{proposition}\label{wev}
For the norming sets   $W_{hi}$ and $W_{ex}$ of the spaces $X_{hi}$ and
$X_{ex}$  we have that
\[ W_{ex|\N_{ev}}=\widehat{W}_{hi}.\]
Morever the space $X_{hi}$ embeds isometrically  into  $X_{ex}$. 
\end{proposition}

%\newpage
%ref{}
\section{A criretion  for the  HI property}
In this section we provide  a criterion for a  Banach space to be an  HI space. It is equivalent to \cite[Lemma 7.1]{CFG}, but we provide a proof for completeness.
\begin{proposition}\label{hi-1}
 Let $X$ be a  Banach space and $X_{0}$ be an  HI subspace of
 $X$. If for every subspace $Z$ of $X$  we have $\dist (S_{X_{0}},
 S_{Z})=0$ then the space $X$ is  HI.
\end{proposition}
\begin{proof}
We show  that  for every two infinite dimensional closed subspaces
$Y_{1}, Y_{2}$  
of $X$ we have  $\dist (S_{Y_{1}},S_{Y_{2}})=0$. Let $\e \in (0,1/10)$.
We start  with the following observation.

In every subspace $Y$ of $X$ we  can find  a Schauder basic sequence $(y_{n})_{n\in\N}\subset Y$ and a normalized sequence $(x_{n})_{n\in\N}\subset S_{X_{0}}$
such that $\norm{y_{n}-x_{n}}<\frac{\e}{2^{n+3}}$. 
Indeed  using Mazur Theorem~\cite{LT}
we get   a normalized Schauder basic sequence  $(u_{n})_{n}$  in $Y$ with basis  constant  $K\leq 2$.
Using the hypothesis by a standard inductive argument we get a block sequence $(y_{n})_{n\in\N}$  of $(u_{n})_{n}$  and
a  normalized sequence $(x_{n})_{n\in\N}\subset S_{X_{0}}$
such that  $\norm{y_{n}-x_{n}}<\frac{\e}{2^{n+3}}$. This implies that $(x_{n})_{n}$ is also Schauder basic sequence.
  
By the above observation  we find   Schauder basic  sequences $(y_{n}^{1})_{n}\subset Y_{1}$,
$(y_{n}^{2})\subset Y_{2}$  and normalized Schauder basic  sequences
$(x_{n}^{1})_{n}, (x_{n}^{2})_{n}$ in $X_{0}$ such that
$\norm{y_{n}^{i}-x_{n}^{i}}<\e/2^{n+3}$   for every $n$ and $i=1,2$.
Since  the subspace $X_{0}$ is HI,  there exist normalized vectors $x_{1}\in
[x_{n}^{1}:n\in\N]$ and
$x_{2}\in [x_{n}^{2}:n\in\N]$ such that  $\norm{x_{1}-x_{2}}<\e/3$.
Since  $(x_{n}^{1})_{n}$ and $(x_{n}^{2})_{n}$ are Schauder  basic
sequences  by $\e/3$ perturbation we may assume that there exists
$n_{1}\in\N$ such that
$x_{1}\in[x_{n}^{1}:n\leq n_{1}]$ and
$x_{2}\in[x_{n}^{2}:n\leq n_{1}]$. Let
$x_{1}=\sum_{i=1}^{n_{1}}a_{n}x_{n}^{1}$ and
$x_{2}=\sum_{i=1}^{n_{1}}b_{n}x_{n}^{2}$.
Then
\begin{equation*}
  \begin{aligned}
  \norm{\sum_{i=1}^{n_{1}}a_{n}y_{n}^{1}-\sum_{i=1}^{n_{1}}b_{n}y_{n}^{2}}&\leq
  \norm{x_{1}-x_{2}}+
  \norm{\sum_{i=1}^{n_{1}}a_{n}y_{n}^{1}-\sum_{i=1}^{n_{1}}a_{n}x_{n}^{1}}+
    \norm{\sum_{i=1}^{n_{1}}b_{n}x_{n}^{2}-\sum_{i=1}^{n_{1}}b_{n}y_{n}^{2}}
    \\ &\leq3\frac{\e}{3}+\max_{n}\abs{a_{n}}\frac{\e}{2^{3}}+\max_{n}\abs{b_{n}}\frac{\e}{2^{3}}<3\e.
  \end{aligned}
\end{equation*}
Using once more that  $\norm{x_{n}^{i}-y_{n}^{i}}<\sfrac{\e}{2^{n+3}}$ 
we get that  $\dist
(S_{Y_{1}},S_{Y_{2}})<10\e$.  This completes the proof that
 $\dist (S_{Y_{1}},S_{Y_{2}})=0$ and therefore $X$ is HI.
\end{proof}

\section{The space  $X_{ex}$ is  HI}\label{xexishi}
To show that the space $X_{ex}$ is  HI
 we  will use
Proposition~\ref{hi-1} and as $X_{0}$ the space $\widehat{X}_{hi}$.   In
order to show that the assumption of the lemma holds 
we will use  the  RISs and the
standard exact pairs as we did for the space  $X_{hi}$.

\subsection{Dependent paired sequences}
In this subsection we shall define the dependent paired sequences which is the analogue
of the  dependent sequences defined for the original space  $X_{ex}$
and we will provide upper estimations for the vectors defined by them.
The definition
of the $(C,\e)$-RIS is the same  as in Definition~\ref{defRIS} with the
obvious modifications.  Similar the definition   of the
$(C,j_{0})$-standard exact pair (SEP) is as in Definition~\ref{defSEP}.

To get upper bounds for the estimation of a  weighted functional on a
$(C,\e)$-RIS
we shall use the basic inequality and  a slight modification of the
auxiliary space.
The  new auxiliary space is the mixed Tsirelson space
$X_{aux2}=T[(\mc{A}_{12n_{j}}, m_{j}^{-1})]_{j\in\N}$.  We denote by
$W_{aux2}$ its norming set.  The choice  of the sequences
$(m_{j})_{j\in\N}, (n_{j})_{j\in\N}$    gives us that on the averages  of
the basis we have
similar estimations as in Lemma~\ref{basisaverages}.
Namely the following holds.
\begin{lemma}\label{basisaverages3}
 Let $x=\frac{m_{j}}{n_{j}}\sum_{k=1}^{n_{j}}e_{k}\in X_{aux2}$. Then  for 
every $f\in W_{aux2}$ with  $w(f)=m_{i}$.
\[
  \abs{f(x)}\leq\begin{cases}\frac{2}{m_{i}}\,\,\text{if $i<j$}\\
        \frac{m_{j}}{m_{i}}\,\,\text{if $i\geq j $}.
  \end{cases}
\]
Moreover  if $f$  has a tree analysis $(f_{\alpha})_{\alpha\in\tree}$ such
that for every weighted $f_{\alpha}$   we have $w(f_{\alpha})\ne m_{j}$  then
$\abs{f(x)}\leq \frac{2}{m^{2}_{j}}$.
\end{lemma}
For a proof of this result see \cite{ATod},Lemma II.9,  (or
\cite{AManStudia},  Lemma 4.2).

The upper bounds  of the action of a weighted functional on averages of a $(C,\e)$-RIS in the space $X_{ex}$ are the same as in Proposition~\ref{eris}.
For sake of completeness we state it again.
\begin{proposition}\label{upperSEP2}
 Let $(x_{k})_{k=1}^{n_{j}}$  be a   $(C,m_{j}^{-2})$-RIS in $X_{ex}$
 and $x=n_{j}^{-1}\sum_{k=1}^{n_{j}}x_{k}$. Then for every $f\in W_{ex}$
 with $w(f)=m_{i}$  we have
 \begin{equation}
   \label{eq:13}
   \abs{f(x)}\leq\begin{cases}
     \frac{3C}{m_{i}m_{j}}\,\,\, &\text{if  $i<j$}\\
     \frac{C}{m_{i}}+\frac{2C}{n_{j}}\,\,\, &\text{if  $i\geq j$}
    \end{cases}.
 \end{equation}
\end{proposition}
 To show that the assumptions of the Lemma~\ref{hi-1} hold we
shall use  the  notion of dependent paired sequences.
\begin{definition}
  A double sequence $(x_{k},f_{k})_{k=1}^{2n_{2j-1}}$ is said to be a $(C,2j-1)$-dependent paired sequence  if there  exist $j_{k},k\leq n_{2j-1}$ such that
  \begin{enumerate}
  \item $x_{2k}=\frac{m_{2j_{k}}}{n_{2j_{k}}}\sum_{i\in F_{k}}e_{i}$, $\# F_{k}=n_{2j_{k}}$ and $F_k\subset\N_{ev}$, and $f_{2k}=m_{2j_{k}}^{-1}\sum_{i\in F_{k}}e_{i}^{*}$.
  \item  $(x_{2k-1},f_{2k-1})$ is a $(C,2j_{k})$-standard exact pair with $\supp(f_{2k-1})\subset\N_{od}$ for every $k$.
\item  $(f_{i})_{i=1}^{2n_{2j-1}}$ is a $(2j-1)$-special sequence of $W_{ex}$ with $w(f_{2k-1})=m_{2j_{k}}=w(f_{2k})$ for every $k$.
  \end{enumerate}
\end{definition}
Let us note that if $(x_{k},f_{k})_{k=1}^{2n_{2j-1}}$ is a  $(C,2j-1)$-dependent paired sequence
then the functional
$F=m_{2j-1}^{-1}\sum_{k=1}^{n_{2n-1}}(f_{2k-1}+f_{2k})\in W_{ex}$ and
\[
\norm{ \frac{m_{2j-1}}{n_{2j-1}}\sum_{k=1}(x_{2k-1}+x_{2k})}\geq
  F(\frac{m_{2j-1}}{n_{2j-1}}\sum_{k=1}(x_{2k-1}+x_{2k}))=2
\]
while  $F(n_{2j-1}^{-1}\sum_{k=1}(x_{2k-1}-x_{2k}))=0$. Our aim is to
show that there exists a $K$ such that for every $j$ and $(C,2j-1)$-dependent paired sequence we have
\[\norm{\frac{m_{2j-1}}{n_{2j-1}}\sum_{k=1}^{n_{2j-1}}(x_{2k-1}-x_{2k})}\leq
  KCm_{2j-1}^{-1}.
  \]
 For this  we    employ  a
variant of the basic inequality.
\begin{proposition}[A variant  of the basic inequality]\label{vbin}
  Let $(x_{k},f_{k})_{k=1}^{2n_{2j-1}}$ be a $(C,2j-1)$-dependent paired sequence and $\lambda_{k},k\leq 2n_{2j-1}$ be scalars.
 Then for every weighted functional $f\in W_{ex}$   we can find  $g_{1},g_{2}\in c_{00}(\N)$
  having non-negative coordinates such that   $\norm{g_{2}}_{\infty}<m_{2j-1}^{-3}$
  and  $g_{1}=h_{1}$ or $g_{1}=e^{*}_{k}+h_{1}$  or $g_{1}=e^{*}_{k}+e^{*}_{k+1}+h_{1}$
  for some  $h_{1}\in W_{aux2}$ with $w(f)=w(h_{1})$  and  $k\notin \supp(h_{1})$  if
  $g_{1}=e^{*}_{k}+h_{1}$ or $k,k+1\notin\supp (h_{1})$ if $g_{1}=e^{*}_{k}+e^{*}_{k+1}+h_{1}$,   satisfying
  \begin{equation}
    \label{eq:3}
    \abs{f(\sum_{k=1}^{2n_{2j-1}}\lambda_{k}x_{k})}
    \leq C(g_{1}+g_{2})\left(\sum_{k=1}^{2n_{2j-1}}\lambda_{k}e_{k}\right).
  \end{equation}
  If we additionally assume that there exist $j_{0}\in\N$ such that for all weighted functionals $g\in W_{ex}$  with $w(g)=m_{j_{0}}$ and every interval $E$ of $\N$ we have that
  \begin{equation}
    \label{eq:10}
  \abs{g(\sum_{k\in E}\lambda_{k}x_{k})}\leq C(\max_{k}\abs{\lambda_{k}}+m_{2j-1}^{-3}\sum_{k\in E}\abs{\lambda_{k}})  
\end{equation}
then if $w(f)\ne m_{j_{0}}$ we may select $h_{1}$ to have a tree analysis  $(h_{t})_{t\in \mc{T}}$  with $w(h_{t})\ne m_{j_{0}}$ for all $t\in\mc{T}$.
\end{proposition}
The proof is identical to the original  one with the only exception  that  a functional with weight  in the interval
$[m_{2j_k}, m_{2j_{k+1}})$ ``can  norm''  at most two  vectors, namely $x_{2k-1},x_{2k}$, instead  of  one in the  case  of the original basic inequality.

As a corollary of the basic inequality  we get the following.
\begin{corollary}\label{x-y}
Let  $(x_{k},f_{k})_{k=1}^{2n_{2j-1}}$  be  a $(C,2j-1)$-dependent paired sequence.  Then
\begin{equation}
  \label{eq:11}
  \norm{\frac{m_{2j-1}}{n_{2j-1}}\sum_{k=1}^{n_{2j-1}}(x_{2k-1}-x_{2k})}\leq \frac{12C}{m_{2j-1}}.
\end{equation}
\end{corollary}
The proof of the corollary follows from Proposition~\ref{vbin} and
Lemma~\ref{basisaverages3} once it is proved
that for every $(2j-1)$-special sequence $(g_{k})_{k=1}^{d}$ of
$W_{ex}$ we have
\begin{equation}
  \label{eq:14}
m_{2j-1}^{-1}\sum_{k=1}^{d}g_{k}(\sum_{k\in E}(-1)^{k+1}x_{k})\leq
  6C (1+ (\#E) m_{2j-1}^{-2}).
\end{equation}
The  proof of this  inequality is based on two key  ingredients which are the same as in the proof of Proposition~\ref{smallnorm}. The
first one is that  for every $k$   the estimation of a weighted 
functional on $x_{2k-1}$ or $x_{2k}$  satisfies upper  bounds which
depend on the weight of the functional (see
Lemma~\ref{basisaverages3} and  Proposition~\ref{upperSEP2}) and
the  second  one  is that the   family of  paired  special sequences has a  tree structure (see  Remark~\ref{treestructure}b).

Using these two ingredients it is shown
that for every $(C,2j-1)$-dependent paired sequence  $(x_{k},f_{k})_{k=1}^{2n_{2j-1}}$ and for every
 $(2j-1)$-special sequence $(g_{k})_{k=1}^{d}$ of
$W_{ex}$ \eqref{eq:14} holds. This implies that the additional part of the variant of the basic inequality holds and using 
Lemma~\ref{basisaverages3} 
we get  \eqref{eq:11}.

\subsection{The  HI property of $X_{ex}$}
We start with the folowing lemma.
\begin{lemma}\label{dist1}
 Let $\e>0$ and $x\in X_{ex}$ be a normalized block vector such that
 $\dist(x, \widehat{X}_{hi})>\e$.  Then for every  $\delta\in
 (0,\e)$ there exists $f\in W_{ex}$  such that
 \[
a)\, \supp(f)\subset \N_{od},\hspace{3cm} b) \, f(x)>\e-\delta.
 \]
\end{lemma}
\begin{proof}
Since  $\dist(x, \widehat{X}_{hi})>\e$ the   Hahn-Banach theorem yields that there exists $g\in
S_{X^{*}_{ex}}$   such that  $g(x)>\e$ and $g_{|\widehat{X}_{hi}}=0$.
Denoting by $E$ the range of $x$  we get that the functional
$F=Eg$ satisfies also that  $F(x)>\e$  and $F(e_{2n})=0$ for
every $n\in\N$. Using that $B_{X^{*}_{ex}}=\overline{W_{ex}}^{\norm{\,}}$  we
get  $x^{*}\in W_{ex}$ such that
\begin{equation}
  \label{eq:1}
\norm{F-x^{*}}<\delta_{1}\,\,\,\  \text{where  $\delta_{1}$
  satisfies}\,\,\,\delta_{1}<\frac{\delta(1-\e)}{(1-\delta)(1+\#E)}.  
\end{equation}
The basis of $X_{ex}$ is bimonotone hence
$\abs{x^{*}(e_{2n})}=\abs{x^{*}(e_{2n})-F(e_{2n})}<\delta_{1}$
and  also we may assume that $x^{*}=Ex^{*}$.
Assuming  by a small perturbation  if necessary, that  $\delta$ is rational we set
\begin{equation}
  \label{eq:15}
f=(1-\delta)x^{*}+(1-\delta)\sum_{2n\in
  E}\abs{x^{*}(e_{2n})}\e_{2n}e^{*}_{2n}
\end{equation}
where $\e_{2n}=-sign (x^{*}(e_{2n}))$.
Note that   $f(e_{2n})=0$ for every $n$.
Indeed  for  every  $n$ such that  $2n\notin E$ we have $x^{*}(e_{2n})=Ex^{*}(e_{2n})=0$
while  if $2n\in E$,
\[
x^{*}(e_{2n})+\abs{x^{*}(e_{2n})}\e_{n}e^{*}_{2n}(e_{2n})=
x^{*}(e_{2n})-x^{*}(e_{2n})=0.
\]

Also $f$ belongs to the rational
convex hull  of $W_{ex}$.
This follows from the fact that $x^{*}(e_{n})\in \Q$ for every $n\in\N$
and  that
\begin{equation}
  \label{eq:2}
\sum_{2n\in E}\abs{Ex^{*}(e_{2n})}\leq\sum_{2n\in E}\norm{Ex^{*}-F}\leq \# E\,\delta_{1}<\delta.  
\end{equation}
We conclude that  $f\in W_{ex}$ since  $W_{ex}$ is rationally convex.
To show that $f$ is the functional we are seeking it remains to show that $f(x)>\e-\delta$. Using \eqref{eq:2}  we
get
\begin{equation*}
  \begin{aligned} f(x)&=(1-\delta)x^{*}(x)+
    (1-\delta)\sum_{2n\in  E}\abs{x^{*}(e_{2n})}\e_{2n}e^{*}_{2n}(x)\\
   &\geq
   (1-\delta)F(x)-(1-\delta)\norm{F-x^{*}}\cdot \norm{x}-(1-\delta)\delta_{1}\# E\\
   &\geq
   (1-\delta)\e-(1-\delta)\cdot\delta_{1}(1+\# E)>\e-\delta.
\end{aligned}
\end{equation*}
\end{proof}
\begin{proposition}\label{hi-2}
  For  every closed infinite dimensional subspace $Y$ of $X_{ex}$ we have that $\dist(S_{Y},S_{\widehat{X}_{hi}})=0$.
\end{proposition}
\begin{proof}
  On the contrary  assume that
 $\dist(S_{Y},S_{\widehat{X}_{hi}})>\e$
for   an infinite dimensional subspace  $Y$.
    We may assume  that  $Y$  is a block subspace.

  We first show that  for every  $j, p\in\N$ there exists  a  $(6/\e, 2j)$- SEP
$(y,f)$ with $y\in Y$, $\min y>p$ and $\supp(f)\subset\N_{od}$.

Since  $\dist(S_{Y},S_{\widehat{X}_{hi}})>\e$ we may choose a
 RIS $(y_{k})_{k\in\N}$  of normalized $2-\ell_{1}^{n_{2j_{k}}}$-averages  in
 $Y$ such that $\dist(y_{k},S_{\widehat{X}_{hi}})>\e$ for every $k$.
  Form Lemma~\ref{dist1}  we get a sequence  of functionals $(f_{k})_{k}\subset W_{ex}$ such that
  \begin{equation}\label{specialpaired1}
   f_{k}(y_{k})>\frac{\e}{2},\quad\,\,\ran(f_{k})\subset \ran(y_{k})\,\,\,\text{and}\,\,\, f_{k|\widehat{X}_{hi}}=0.
  \end{equation}
  For every $k$ choose $c_{k}\in [1,2/\e)$ such that $f_{k}(c_{k}y_{k})=1$.  Then for every $j\in\N$ and $F_{j}\subset\N$ with $\# F_{j}=n_{2j}$ and $\min F_{j}> \max\{n_{2j},p\}$ it  follows that the pair $(y,f)$ where
  \begin{equation}
    \label{eq:12}
     y=\frac{m_{2j}}{n_{2j}}\sum_{k\in F_{j}}c_{k}y_{k}\quad\text{and}\,\,\, f=m_{2j}^{-1}\sum_{k\in F_{j}}f_{k}\quad
\end{equation}
   is a $(6/\e, 2j)$- SEP  (see  Definition~\ref{defSEP})
with $y\in Y$, $\min\supp(y)>p$ and $\supp(f)\subset \N_{od}$.

Next for every $j\in\N$ we defined
a $(6\e^{-1},2j-1)$-dependent paired sequence   $(x_{k},f_{k})_{k=1}^{2n_{2j-1}}$
as follows:

Choose  $j_{1}$ such that $n_{2j-1}^{2}<m_{2j_1}$. Let $(x_{1},f_{1})$ be  a
$(6\e^{-1},2j_{1})$-SEP as in \eqref{eq:12}. 
Since $f_{1}\in W_{ex}$ we have that 
$f_{1}\in V_{n_{1}}$ for some $n_{1}$. Using the coding function $\varrho_{n_{1}}^{2j_{1}}$
we get the functional
$g_{2}=\varrho^{2j_{1}}_{n_{1}}(f_{1})=m_{2j_{1}}^{-1}\sum_{k\in B_{1}}e_{k}^{*}\in K_{n_{1}}^{2j_{1}}$.  We set
\[f_{2}=\widehat{g_{2}}\quad \text{and} \quad  x_{2}=\frac{m_{2j_{1}}}{n_{2j_{1}}}\sum_{k\in B_{1}}e_{2k}.
\]
Let $2j_{2}=\sigma(g_{2})$  where $\sigma$ is the coding function  used to define the special sequences in the space $X_{hi}$.
Choose  $(x_{3},f_{3})$   a
$(6\e^{-1},2j_{2})$-SEP as in \eqref{eq:12} satisfying $\supp(x_{2})<\supp (x_{3})$.  This yields that also  $\supp (f_{2})<\supp(f_{3})$. 
Let $n_{3}$ be such that 
$f_{3}\in V_{n_{3}}$. Using the coding function $\varrho_{n_{3}}^{2j_{2}}$
we get the functional
$g_{4}=\varrho^{2j_{2}}_{n_{3}}(f_{3})=m_{2j_{2}}^{-1}\sum_{k\in B_{3}}e_{k}^{*}\in K_{n_{3}}^{2j_{2}}$.  We set
\[f_{4}=\widehat{g_{4}}\quad \text{and} \quad  x_{4}=\frac{m_{2j_{2}}}{n_{2j_{2}}}\sum_{k\in B_{3}}e_{2k}.
\]
It is clear now how we proceed to get
a $(6\e^{-1},2j-1)$-dependent paired sequence.

Set $y=\frac{m_{2j-1}}{n_{2j-1}}\sum_{k=1}^{n_{2j-1}}x_{2k-1}$ and
$x=\frac{m_{2j-1}}{n_{2j-1}}\sum_{k=1}^{n_{2j-1}}x_{2k}$. From the choice of the sequence we have that $y\in Y$ and  $x\in \widehat{X}_{hi}$.
Corollary~\ref{x-y} yields that
$\norm{y-x}<72\e^{-1}m_{2j-1}^{-1}$.
 Also since   $(x_{2k-1})_{k=1}^{n_{2j-1}}$ is a $(2,
 m_{2j-1}^{-2})$-RIS  it follows that  $1\leq\norm{x}\leq4$. This
 yields $\dist (S_{Y}, S_{\widehat{X}_{hi}})<2\cdot
 72C\e^{-1}m_{2j-1}^{-1}$.   It follows that
 \[
 \frac{\e}{2}< \frac{72}{\e \cdot m_{2j-1}}\Leftrightarrow m_{2j-1}\e^{2}<2\cdot 72 \]
a contradiction for large enouge $j$. Therefore
$\dist (S_{Y}, S_{\widehat{X}_{hi}})=0$ for every infinite dimensional
subspace $Y$  of $X_{ex}$.
\end{proof}
The above yields the following.
\begin{theorem}
  The  HI   space $X_{hi}$ is  isometric to an infinite codimensional
  subspace of  the  HI space
  $X_{ex}$.
 \end{theorem}
 \begin{proof}
From Proposition~\ref{wev} we have that
$W_{ex|\N_{ev}}=\widehat{W}_{hi}$ and this shows that  the
   space $\widehat{X_{hi}}=[e_{2n}:n\in\N]$ is isometric to $X_{hi}$.
   Propositions~\ref{hi-1} and \ref{hi-2} yield that the space $X_{ex}$
   is HI and this completes the proof.
 \end{proof}
 \section{The  HI extension of a variant of $X_{hi}$.}
There are several  examples of HI spaces where the  norming set  W
does not
include convex combinations and this is necessary for the restricted
saturation method applied for the definition of $W$, see \cite{AMot},\cite{AMot2}.
Also the norming set in the Argyros-Deliyanni asymptotic $\ell_{1}$-HI
space is  defined by a norming set $W$ which does not include rational
convex combinations.
Since the  closedness  in rational convex combinations  is   essential to show that the  extended space  is HI, naturally this raises the question if the spaces defined
by a norming set  not closed  in  rational convex combinations admit an   HI extension. In this section we provide a positive answer to this problem for a variant  of the space $X_{hi}$ used before, which is defined by a norming set $W_{vhi}$ which is similar to  $W_{hi}$ but not closed in rational convex combinations.  More precisely  we consider the following norming set.
Let
  $\sigma: \Q_{s}\to \{2j: j\in\Omega_{2}\}$ be  a  coding function 
satisfying \eqref{eq:16}.
\begin{definition}
  The set  $W_{vhi}$  is the minimal subset of $c_{00}$  satisfying the following  conditions:
  \begin{enumerate}
  \item  $\{\pm e_{n}: n\in\N\}\subset W_{vhi}$.
  \item $W_{vhi}$ is symmetric and is closed under the restriction of its elements on intervals of $\N$.
  \item  $W_{vhi}$ is closed under the $(\mc{A}_{n_{2j}},m_{2j}^{-1})$-operations, i.e.,
     $m_{2j}^{-1}\sum_{i=1}^{d}f_{i}\in W_{vhi}$  for all $n_{2j}$-admissible sequences $(f_{i})_{i=1}^{d}$ of $W_{vhi}$.
  \item  $W_{vhi}$ is closed under the $(\mc{A}_{n_{2j-1}},m_{2j-1}^{-1})$-operations on special sequences, i.e.,
    $m_{2j-1}^{-1}\sum_{i=1}^{d}f_{i}\in W_{vhi}$  for all $n_{2j-1}$-admissible sequences $(f_{i})_{i=1}^{d}$  of $W_{vhi}$ such that  $w(f_{1})=m_{2j_{1}}>n_{2j-1}^{2}$ and $w(f_{i+1})=m_{\sigma(f_{1},\dots, f_{i})}$ for $1\leq i\leq d-1$.
  \end{enumerate}
\end{definition}
The difference between   $W_{vhi}$ and $W_{hi}$ is the  rational convexity.
It follows by the same arguments as in section~\ref{standardhi} that the space $X_{vhi}$
with the  set $W_{vhi}$ as its norming set is also HI.

We define  a variant $W_{vex}$ of the extended   norming set $W_{ex}$. The definition of $W_{vex}$ is very similar to the previous  one $W_{ex}$ and goes as follows:

Let $\widetilde{W}_{0}=\{\pm e_{n}^{*}:n\in\N\}$ and assume that $\widetilde{W}_{n-1}$  has been defined.
We set 
 \[
   \widetilde{U}_{n}=\cup_{j}\{m_{2j}^{-1}\sum_{i=1}^{k}f_{i}: (f_{i})_{i=1}^{k}\;\text{ is an }
    \mc{A}_{n_{2j}}\text{-admissible sequence of}\,\, \widetilde{W}_{n-1}\},
 \]
 and $ \widetilde{V}_{n}=( \widetilde{U}_{n}\setminus\cup_{j=0}^{\infty}\cup_{k=n+1}^{+\infty}K_{k}^{2j})\cup  \widetilde{W}_{n-1}$.

For $j\in\N$ we  define
\begin{align}
\widetilde{W}_{n}^{2j-1}&=\{\frac{1}{m_{2j-1}}\left(\sum_{i=1}^{l}(f_{i}+\widehat{g}_{i})+
    \sum_{i=l+1}^{m}\widehat{g}_{i}+\sum_{i=m+1}^{p}h_{i}\right):
 \label{eq3a}
  \\
 &
 \resizebox{0.92\hsize}{!}{$
  (f_{1},\widehat{g}_{1},\dots,f_{l},\widehat{g}_{l},
  \widehat{g}_{l+1},\dots\widehat{g}_{m},h_{m+1},\dots,h_{p})\,
\text{is a}\,\,
   (2j-1,  \widetilde{V}_{n},(\varrho^{2k}_{i})_{i=1}^{n}, k=0,1,\dots)$}
\notag
  \\
&\hspace{7cm}\text{special sequence}
\}.
                                                   \notag
\end{align}
where the definition  of a  $(2j-1,  \widetilde{V}_{n},(\varrho^{2k}_{i})_{i=1}^{n}, k=0,1,\dots)$ special sequence is identical to the corresponding definition   in $W_{ex}$.
We set
\begin{equation*}
  \begin{split}
    \widetilde{W}_{n}=
\cup_{j}\widetilde{W}_{n}^{2j-1}\cup \widetilde{V}_{n}\cup
\{\sum_{i=1}^{d}r_{i}f_{i}:& (r_{i})_{i=1}^{d}  \text{is a rational convex combination},\\
& f_{i}\in \cup_{j}\widetilde{W}_{n}^{2j-1}\cup \widetilde{V}_{n}\,\,
\,\text{and}\,\,\, \supp (\sum_{i=1}^{d}r_{i}f_{i})\subset\N_{od}
  \}.
  \end{split}
\end{equation*}
The norming set $W_{vex}$ is the set $W_{vex}=\cup_{n=0}^{+\infty}\widetilde{W}_{n}$.

 Lemmas~\ref{step1} and \ref{step2}  remain valid for the norming sets
 $W_{vhi}$ and $W_{vex}$ and the proofs are identical. These yield that  $X_{vhi}$ is isometric to the subspace  $\widehat{X}_{vhi}=[e_{2n}:n\in\N]$  of $X_{vex}$.
 We have to prove that  $X_{vex}$ is HI.
 To prove this  first we note that  the  variant of the basic inequality, Proposition~\ref{vbin} and 
Corollary~\ref{x-y} hold also for the space $X_{vex}$. Moreover  the following holds.
\begin{proposition}
 The space  $X_{vex}$ is reflexive  and $B_{X_{vex}^{*}}=\overline{\mathrm{conv}}^{\norm{\,} }(W_{vex})$ 
\end{proposition}
To prove that $X_{vex}$ is  HI  we need  a slight modification in the proof
of Lemma~\ref{dist1} since in that  proof we use that $B_{X^{*}_{ex}}=\overline{W_{ex}}^{\norm{\cdot}}$ which does not hold for the space  $X_{vex}$.
\begin{lemma}\label{dist2}
 Let $\e>0$ and $x\in X_{vex}$ be a normalized block vector such that
 $\dist(x, \widehat{X}_{vhi})>\e$.  Then for every $\delta\in
 (0,\e)$ there exists $f\in W_{vex}$  such that
 \[
a)\, \supp(f)\subset \N_{od},\hspace{3cm} b) \, f(x)>\e-\delta.
 \]
\end{lemma}
\begin{proof}
  Since  $\dist(x, \widehat{X}_{vhi})>\e$ there exists $f\in
S_{X^{*}_{vex}}$   such that  $g(x)>\e$ and $g_{|\widehat{X}_{vhi}}=0$.
Denoting by $E$ the range of $x$  we get that the functional
$F=Eg$ satisfies also that  $F(x)>\e$  and $F(e_{2n})=0$ for
every $n\in\N$. Using that $B_{X^{*}_{vex}}=\overline{\text{conv}}(W_{vex})^{\norm{\,}}$  we
get  $x^{*}\in \conv_{\Q}(W_{vex})$ such that
\begin{equation}
  \label{eq:1a}
\norm{F-x^{*}}<\delta_{1}\,\,\,\  \text{where  $\delta_{1}$
  satisfies}\,\,\,\delta_{1}<\frac{\delta(1-\e)}{(1-\delta)(1+\#E)}.  
\end{equation}
The basis of $X_{vex}$ is bimonotone hence
$\abs{x^{*}(e_{2n})}=\abs{x^{*}(e_{2n})-F(e_{2n})}<\delta_{1}$ and 
also we may assume that  $x^{*}$
is  supported by the set $E$, i.e.,  $x^{*}=Ex^{*}$.

Assuming by a small perturbation if necessary that  $\delta $ is rational
we set 
\begin{equation}
  \label{eq:15a}
f=(1-\delta)x^{*}+(1-\delta)\sum_{2n\in
  E}\abs{x^{*}(e_{2n})}\e_{2n}e^{*}_{2n}\,\,\text{where}\,\,\e_{2n}=-sign (x^{*}(e_{2n})).
\end{equation}
As in Lemma~\ref{dist1} we get that
that $f\in \conv_{\Q}(W_{vex})$ and  $\supp (f)\subset\N_{od}$.

Since   $W_{vex}$ is closed  in the rational convex combinations whose support is subset of $\N_{od}$ it follows that   $f\in W_{vex}$.
The rest of the proof is identical  with the  proof of Lemma~\ref{dist1}.
\end{proof}
Using the above  lemma  we get the following.
\begin{proposition}
The space $X_{vex}$ is HI.  
\end{proposition}
The proof of the proposition  is identical to the proof of  the Proposition~\ref{hi-2}.
\begin{remark}
  The extension method   applied   for the  HI spaces defined by the
  families $\mc{A}_{n}$ can also be applied with identical proofs
  for the Gowers-Maurey space \cite{GM}.
\end{remark}
\section{The case   of asymptotic  $\ell_{1}$-HI Banach spaces}
In this section we show that the extension  method can be applied also
to asymptotic  $\ell_{1}$-HI Banach spaces, see \cite{AD}. The method is
identical to the one for the spaces defined  by the families
$\mc{A}_{n_{j}}$ with some slight modifications due to
the usage of the Schreier families.

Let us recall the definition of the norming set of a standard
asymptotic $\ell_{1}$-HI space.

Let $(n_{j})_{j=0}^{\infty}, (m_{j})_{j=0}^{\infty}$ be two sequences of positive
integers
such that $m_{0}=2$,  $m_{j+1}=m_{j}^{m^{j}}$,  $n_{0}=1$  and
$n_{j+1}=m_{1}^{2m_{j+1}}n_{j}$. Let also $\sigma$ be a coding function defined
as in  Section~\ref{standardhi}.

The norming set  $W_{hi}$ of a standard asymptotic $\ell_{1}$-HI space
is the minimal subset of $c_{00}(\N)$  satisfying the following
properties:
\begin{enumerate}
\item[i)]  It is   rationally convex, symmetric, closed on the projections on
  intervals  of $\N$ and contains $\pm e_{n}^{*}$, $n\in\N$.
  \item[ii)]  It is closed  in $(\mc{S}_{n_{2j}},m_{2j}^{-1})$-operations, i.e., if $(f_{i})_{i\in F}$ is an $\mc{S}_{n_{2j}}$-admissible
 sequence of $W_{hi}$  then
 $m_{2j}^{-1}\sum_{i\in F}f_{i}\in W_{hi}$.
\item[iii)] It is  closed  in
  $(\mc{S}_{n_{2j-1}},m_{2j-1}^{-1})$-operations on special sequences, i.e., if $(f_{i})_{i\in F}$ is an $\mc{S}_{n_{2j}}$-admissible
 sequence of $W_{hi}$  $w(f_{i})=m_{2j_{i}}=\sigma(f_{1},\dots,f_{i-1})$
 and $m_{2j_{1}}\geq n_{2j-1}^{2}$
 then
 $m_{2j-1}^{-1}\sum_{i\in F}f_{i}\in W_{hi}$.
\end{enumerate}
To define the extended  HI space we shall use a variation of the
Schreier families $(\mc{S}_{n})_{n\in\N}$, (see \cite{ATolmemoir} for
some of their properties).
\begin{definition}
  We set $\mc{S}_{1}^{f}=\{\{1\}\}\cup \{F\subset\N: 2\#F\leq min F\}$
  and
  \[
    \mc{S}_{n}^{f}=\mc{S}_{1}^{f}[\mc{S}_{n-1}^{f}]=\{
\cup_{i=1}^{k}F_{i}:  F_{1}<\dots<F_{k}\in\mc{S}_{n-1}^{f}, \{\min F_{i}:i\leq
k\}\in\mc{S}_{1}^{f}\}.
  \]
\end{definition}
We also the define the families
\[
  \mc{S}_{n}^{f}\odot \mc{A}_{2}=\{\cup_{i=1}^{d}F_{i}:
F_{1}<\dots<F_{d},\, \# F_{i}\leq 2\,\,\text{and}\,\{\max F_{i}\}_{i=1}^{d}\in\mc{S}_{n}^{f}\}
.
\]
It is easy to see that  for every $n$  the families  $\mc{S}_{m}^{f},{S}_{n}^{f}\odot\mc{A}_{2}$ are
regular.

For a finite subset   $G=\{n_{1}<\dots<n_{d}\}$ of $\N$ we denote by $\widehat{G}$
the set $\widehat{G}=\{2n_{1}<\dots <\dots<2n_{d}\}$.  We also  denote by
$\widehat{\mc{S}_{n}}$ the family
$
\{\widehat{F}: F\in\mc{S}_{n}\}.
$
\begin{lemma}\label{lsa2}
For every  $n\in\N$ we have
\[\widehat{\mc{S}_{n}}= \mc{S}_{n}^{f}\cap[\N_{ev}]^{<\omega}.\]
\end{lemma}
\begin{proof}
First we show by induction  on $n$  that 
  \[
\mc{S}_{n}^{f}\cap[\N_{ev}]^{<\omega}\subset \widehat{\mc{S}_{n}}.
  \]
  Let    $F=\{2n_{1}<2n_{2}<\dots<2n_{d}\}\in\mc{S}_{1}^{f}$. Then
  $2d\leq 2n_{1}\Rightarrow d\leq n_{1}$.
  It follows that
  $ G=\{n_{1}<n_{2}<\dots<n_{d}\}\in\mc{S}_{1}$ and $F=\widehat{G}$.

  Assume that the result holds for every $k\leq n$ and
  $F\in \mc{S}_{n+1}^{f}\cap[\N_{ev}]^{<\omega}$.  Then
  $F=\cup_{i=1}^{d}F_{i}$  with $F_{i}\in\mc{S}_{n}^{f}$  and  $\{\min
  F_{i}:i\leq d\}\in\mc{S}_{1}^{f}$. By the inductive hypothesis  we have
  for every $i\leq d$ there exits  $G_{i}\in\mc{S}_{n}$ such that
  $F_{i}=\widehat{G}_{i}$.  From the case  $n=1$ we get that $\{\min
  G_{i}:i\leq d\}\in\mc{S}_{1}$ and therefore
  $G=\cup_{i=1}^{d}G_{i}\in\mc{S}_{n+1}$.  Also  $\widehat{G}=F$ and this
  completes  the  proof of the induction. 

  The  inclusion  $ \widehat{\mc{S}_{n}}\subset\mc{S}_{n}^{f}\cap[\N_{ev}]^{<\omega}$ is
  proved similarly.
\end{proof}
We give the definition of the admissibility with  respect  the
families  $\mc{S}_{n}$ and $\mc{S}_{n}^{f}\odot\mc{A}_{2}$. 
\begin{definition}\label{sa2}
  We  call a block sequence $(x_{i})_{i=1}^{d}$ in $c_{00}(\N)$
  \begin{enumerate}
  \item  $\mc{S}_{n}$-admissible if $\min\{\supp (x_{i}):i\leq
    d\}\in\mc{S}_{n}$.
      \item  $\mc{S}^{f}_{n}$-admissible if $\min\{\supp (x_{i}):i\leq  d\}\in\mc{S}^{f}_{n}$.
\item $\mc{S}^{f}_{n}\odot\mc{A}_{2}$-admissible   if $\{\min\supp
  (x_{i})_{i=1}^{d}\}=\cup_{i=1}^{m}F_{i}$, $F_{i}<F_{i+1}$ , $\#F_{i}\leq2$
for every $i$ and $\{\max F_{i}:i\leq m\}\in\mc{S}_{n}^{f}$.
    \end{enumerate}
  \end{definition}
Let us mention that admissibility with respect  the  families
$\mc{S}_{n},\mc{S}_{n}^{f}, n\in\N$ is the standard one, see \cite{ATod},
\cite{ATolmemoir}, while  the admissibility with respect  the  families
$\mc{S}_{n}^{f}\odot\mc{A},n\in\N$ is a variation of the standard one,
suitable for our  applications.
\begin{remark}\label{rsa2}
Let  $G=\{m_{1}<m_{2}<\dots<m_{d}\}\in\mc{S}_{n}$. 
For every  $0\leq j\leq d$ and every $k_{1}<2m_{1}<k_{2}<2m_{2}<\dots<
k_{j}<2m_{j}$
we have that the set
 $\{
 k_{1},2m_{1},\dots, k_{j},2m_{j},2m_{j+1},2m_{j+2},\dots, 2m_{d}\}$
 is   $\mc{S}_{n}^{f}\odot\mc{A}_{2}$-admissible.

 Indeed if $F_{i}=\{k_{i},2m_{i}\}$ for $i\leq j$ and $F_{i}=\{2m_{i}\}$
 for $i>j$,
 then Lemma~\ref{lsa2} yields that $\{\max F_{i}:i\leq d\}$ is
 $\mc{S}_{n}^{f}$-admissible and hence $\cup_{i=1}^{d}F_{i}$ is
 $\mc{S}_{n}^{f}\odot\mc{A}_{2}$-admissible (see Definition~\ref{sa2}).
 \end{remark}

 The norming set of the extended space will be an EO-complete subset of
 the norming set $W_{mT}$
 of the mixed Tsirelson space 
  $T[(\mc{S}_{n_{2j}}^{f},m_{2j}^{-1})_{j\in\N}, (\mc{S}_{n_{2j-1}}^{f}\odot\mc{A}_{2},
  m_{2j-1}^{-1})_{j\in\N}]$.

 In order to define the extended norming set
as
in Subsection~\ref{EOcomplete}
we consider  the sets
\[K^{2j}=\{m_{2j}^{-1}\sum_{i\in F}e^{*}_{i}: F\,\,\text{is maximal
    $\mc{S}_{n_{2j}}$-set}\}\]
and  a partition $(K_{i}^{2j})_{i}$   of them such that for every $i$,
$\sup\{\min(f):f\in
K_{i}^{2j}\}=+\infty$.  The sets $W_{mT}^{2j}$ 
and  the coding functions  $\varrho_{i}^{2j}:W_{mT}^{2j}\to K_{i}^{2j}$ are also
defined as in Subsection~\ref{EOcomplete}.

The definition of the paired special, semi paired special and special
sequences
are   the same  using
$\mc{S}_{n_{2j-1}}^{f}\odot\mc{A}_{2}$-admissibility instead of the  $\mc{A}_{n_{2j}}$-admissibility.
Let us give the definition of the 
$(2j-1,V,(\varrho_{k}^{2i})_{k=1}^{n}, i=0,1,\dots)$- paired special
sequence.
\begin{definition}\label{pairedspecials2}
  Let $j,n\in\N$ and $V\subset W_{mT}$.
We call 
$(2j-1,V,(\varrho_{k}^{2i})_{k=1}^{n}, i=0,1,\dots)$- paired special
sequence
any block sequence
$(f_{1},\widehat{g}_{1},\dots, f_{p},\widehat{g}_{p})$ satisfying
the following:
\begin{enumerate}
\item   it is $\mc{S}_{n_{2j-1}}^{f}\odot\mc{A}_{2}$-admissible  and for every $k$ there exists  $j_{k}\in\N$  such
  that  $f_{k}\in V\cap W_{mT}^{2j_{k}}$  and  $\widehat{g}_{k}=\widehat{\varrho^{{2j_{k}}}_{i_{k}}(f_{k})}$  for some $i_{k}\leq n$.
\item   $(g_{1},g_{2},\dots,g_{p})$ is a $(2j-1)$-special
  sequence of $W_{hi}$.
\end{enumerate}
\end{definition}
\begin{remark}\label{rsa2a}
  Let us  point out   that  the
  $\mc{S}_{n_{2j-1}}$-admissibility 
of $(g_{1},g_{2},\dots,g_{p})$ implies the 
$\mc{S}_{n_{2j-1}}^{f}\odot\mc{A}_{2}$-admissibility of the sequence
$(f_{1},\widehat{g}_{1},\dots, f_{p},\widehat{g}_{p})$ (see
Remark~\ref{rsa2}).
%, implies the $\mc{S}_{n_{2j-1}}$-admissibility 
%of $(g_{1},g_{2},\dots,g_{p})$ be $\mc{S}_{n_{2j-1}}$.
\end{remark}

The definition of the
$(2j-1,V,(\varrho_{k}^{2i})_{k=1}^{n}, i=0,1,\dots)$- paired special
sequences and
$(2j-1,V,(\varrho_{k}^{2i})_{k=1}^{n}, i=0,1,\dots)$- special
sequence are as in Subsection~\ref{EOcomplete}.%  with the only
% difference the usage of
% $\mc{S}_{n_{2j-1}}^{f}[\mc{A}_{2}]$-admissibility as we did in the
% definition of the paired special sequences.

Next we prove that $W_{ex|\N_{ev}}=\widehat{W}_{hi}$ proving the
analogues of
Lemmas~\ref{step1},\ref{step2}.
%also hold   for the new norming sets  $W_{hi}$ and
%$W_{ex}$. The proofs  are identical with the  use of
%Lemma~\ref{lsa2}.
\begin{lemma}\label{step1a}
 For every  $n\in\N$ and every $f\in W_{n}$  we have that  $f_{|\N_{ev}}\in \widehat{W}_{hi}$. 
\end{lemma}
\begin{proof}
As  in Lemma~\ref{step1} we use induction. The case $n=0$ is identical.
 Assume that  the result holds for all  $k\leq n$  and  let $f\in W_{n+1}$.

If  $f$ is   the result of an  $(\mc{S}^{f}_{n_{2j}}, m_{2j}^{-1})$
  operation of elements  of $W_{n}$, i.e., $f=m_{2j}^{-1}\sum_{i=1}^{d}f_{i}$, from the  inductive 
  hypothesis we have that  for every $i$ there exists $\phi_{i}\in W_{hi}$ such
  that $f_{i|\N_{ev}}=\widehat{\phi_{i}}$.
 Lemma~\ref{lsa2}
  yields that $\{\minsupp(\phi_{i}):i\leq d\}$ is
  $\mc{S}_{n_{2j}}$-admissible. Since $W_{hi}$ is closed in
$(\mc{S}_{n_{2j}},m_{2j}^{-1})$-operations
we get that
  $f_{|\N_{ev}}=\widehat{m_{2j}^{-1}\sum_{i=1}^{d}\phi_{i}}\in \widehat{W}_{hi}$.

  The result  in the case
  where  $f$ is the  result of an
$(2j-1, V_{n},
(\varrho_{k}^{2i})_{k=1}^{n},i=0,1,\dots)$-special sequence 
is  analogue  to the previous one
using Remark~\ref{rsa2a}. The case  is  $f$ is  a rational convex
combination  follows readily.
\end{proof}
\begin{lemma}\label{step2a}
For every $f\in W_{hi}$   there exists  $\phi\in W_{ex}$ such that  $\phi_{|\N_{ev}}=\widehat{f}$.
\end{lemma}
\begin{proof}
  Similar to the  proof of Lemma~\ref{step2}  we proceed  by
  induction.

  For $n=1$ and  
  $f=m_{2j}^{-1}\sum_{i\in G}\pm e_{i}^{*}$, $G\in\mc{S}_{n_{2j}}$
  we observe  from Lemma~\ref{lsa2} that $\widehat{G}\in\mc{S}_{n}^{f}$.
  Using this observation, as in Lemma~\ref{step2} we conclude that
  $\widehat{f}\in W_{ex}$

We pass now to the case where
 $f=m_{2j-1}^{-1}\sum_{i=1}^{p}g_{i}$  and  $(g_{1},\dots,g_{p})$ is a
$(2j-1)$-special sequence  of $W_{hi}^{1}$.
As in
Lemma~\ref{step2} the proof is based on the following claim.

\begin{claim} 
  Let  $g_{i}\in K_{n_{i}}^{2j_{i}}$, $i\leq m$  be such that
 $(g_{1},\dots,g_{m})$ is  a $(2j-1)$-special sequence of
  $W_{hi}^{1}$.  Then there exist $0\leq l\leq m$ and $f_{1},\dots,f_{l}\in
  W_{ex}$ such that $(f_{1},\widehat{g}_{1},\dots,f_{l},\widehat{g}_{l},\widehat{g}_{l+1},\dots,\widehat{g}_{m})$
  is a $(2j-1, V_{n}, (\varrho_{k}^{n})_{k=1}^{n},i=0,\dots)$ semi-paired
  special sequence where  $n=\max\{n_{i}:i\leq m\}$.
\end{claim}
\textit{Sketch of the proof}.  Let $l\le m$  be maximal with the properties
\begin{enumerate}
\item $g_{i}\in \varrho_{n_{i}}^{2j_{i}}(V_{n}\cap W_{mT}^{2j_{i}})$ for every $i\leq l$.
\item
$F=((\varrho^{2j_{i}}_{n_{i}})^{-1}(g_{i}), \widehat{g}_{i})_{i=1}^{l}$  is a  $(2j-1, V_{n},
(\varrho_{k}^{2i})_{k=1}^{n},i=0,1\dots)$ -paired special sequence, where
$n=\max\{n_{i}:i\leq m\}$.
\end{enumerate}
For every $i$, as in the first case we have $\widehat{g}_{i}\in W_{ex}$,   and from
Remark~\ref{rsa2}    we get 
$(
\varrho^{2j_{1}}_{n_{i}})^{-1}(g_{1}), \widehat{g}_{1},\dots,
  (\varrho^{2j_{l}}_{n_{i}})^{-1}(g_{i}), \widehat{g}_{l},
  \widehat{g}_{l+1},\dots,\widehat{g}_{m}
  )$
is $\mc{S}_{n_{2j-1}}^{f}\odot\mc{A}_{2}$-admissible in $W_{ex}$.
Considering the cases $l=0,l=m$ and $0<l<m$,   as in the proof
of   Lemma~\ref{step2} we get that
$(\varrho^{2j_{1}}_{n_{i}})^{-1}(g_{1}), \widehat{g}_{1},\dots,
  (\varrho^{2j_{l}}_{n_{i}})^{-1}(g_{i}), \widehat{g}_{l},
  \widehat{g}_{l+1},\dots,\widehat{g}_{m}
  )$
is a $(2j-1,
V_{n},(\varrho_{k}^{2i})_{k=1}^{n},i=0,1,\dots)$-semi paired special
sequence.
$\square$

Returning back to the proof, let   $f=m_{2j-1}^{-1}\sum_{i=1}^{p}g_{i}$ where  $(g_{1},\dots,g_{p})$ is a
$(2j-1)$-special sequence  of $W_{hi}^{1}$,  and  $m\leq p$ be maximal with the property $g_{i}\in K_{2n_{i}}^{2j_{i}}$ for
every $i\leq m$.
It follows from the Claim and the definition of the special sequences that
$((\varrho^{2j_{1}}_{n_{1}})^{-1}(g_{1}), \widehat{g}_{1},\dots,
  (\varrho^{2j_{l}}_{n_{m}})^{-1}(g_{m}), \widehat{g}_{m},
  \widehat{g}_{m+1},\dots,\widehat{g}_{p}
  )$
  is a $(2j-1,
V_{n},(\varrho_{k}^{2i})_{k=1}^{n},i=0,1,\dots)$-special
sequence. Hence
$\phi=m_{2j-1}^{-1}(\sum_{i=1}^{m}
((\varrho^{2j_{l}}_{n_{i}})^{-1}(g_{i})+\widehat{g}_{i})+\sum_{m+1}^{p}\widehat{g}_{i})\in W_{ex}$
and $\widehat{f}=\phi_{|\N_{ev}}$.

Assume that we have proved  the result for every $k\leq n$ and let $f\in
W_{hi}^{n+1}$.

In the case  $f$  is the  result of an
$(\mc{S}_{n_{2j}},m_{2j}^{-1})$-operation  of  elements of $W_{hi}^{n}$, using Lemma~\ref{lsa2}
and the property $W_{ex}$ is closed in the
$(\mc{S}^{f}_{n_{2j}},m_{2j}^{-1})$-operations, we  show,  as in Lemma~\ref{step2}, that $\widehat{f}=\phi_{|\N_{ev}}$ for some $\phi\in W_{ex}$.

In the case   $f=m_{2j-1}^{-1}\sum_{i=1}^{p}g_{i}$ where 
$(g_{1},\dots,g_{p})$ is a $(2j-1)$-special sequence of $W_{hi}^{n+1}$,
is  proved  as in Lemma~\ref{step2}  using the    above Claim.
 \end{proof}
As a corollary we get  the following.
\begin{proposition}\label{wev2}
For the norming sets   $W_{hi}$ and $W_{ex}$ of the spaces $X_{hi}$ and
$X_{ex}$  we have that
\[ W_{ex|\N_{ev}}=\widehat{W}_{hi}.\]
As a consequence the space  $X_{hi}$  is isometric to the subspace of
$X_{ex}$ generated by the sequence $(e_{2n})_{n\in\N}$.
\end{proposition}

The results of the Sections ~\ref{standardhi},\ref{swex},\ref{xexishi}
with the obvious modifications imposed by the usage of the Schreier
families in place of the  $\mc{A}_{n_{j}}$ families, hold  also for the new
spaces with similar proofs (see \cite{ATolmemoir} for a detailed
exposition of the related results).

The above  yield    the following.
\begin{theorem}
 A standard asymptotic $\ell_{1}$  HI Banach space is isometric to an
 infinite codimensional subspace of a reflexive  asymptotic   $\ell_{1}$
 HI Banach space.
\end{theorem}
\begin{remark}
 Applying the analogue modifications we get that any standard
 asymptotic $\ell_{p}$ HI space,
$1<p<+\infty$, (see \cite{DM}),
admits an HI extension.
\end{remark}

\section{Final discussion}

We discuss some open problems that are related to this article. We begin with the one that motivated this paper and remains open in its full generality.

\begin{problem}[Pe\l czy\'nski]
\label{HI extension}
Does every HI space $X$ embed into an HI space $Y$ such that $Y/X$ is infinite dimensional? 
\end{problem}

Problem \ref{HI extension} was posed by Pe\l czynski at a conference where the first author was present, however we were unable to find it in the literature. Yet, he posed a strongly related one in \cite[Problem 3.3, page 348]{P}.

\begin{problem}[Pe\l czynski]
Does every separable HI space $X$ embed into an HI space $Y$ with a Schauder basis?
\end{problem}

It is remarkable that in the same paper \cite{P} Pe\l czy\'nski posed the problem of the existence of a $\mathcal{L}_\infty$-space that is HI, a fact that the first author was unaware of. This is why this paper was not cited in \cite{AH}, where the scalar-plus-compact problem was solved using such a space. Let us point out that the methods developed in the current paper show promise in constructing HI extensions of $\mathcal{L}_\infty$-HI spaces, such as the Argyros-Haydon space.

After sharing a preliminary version of this article to W. B. Johnson, he pointed out to us a more general form of Pe\l czy\'nski's HI extension Problem \ref{HI extension}, based on the following definition (see also \cite[page 4701]{CFG} for a formulation in the setting of twisted sums).

\begin{definition}
Let $X$ be a Banach space $Y$ be a superspace of $X$. We call $Y$ a singular extension of $X$ if $Y/X$ is infinite dimensional and the quotient map $Q:Y\to Y/X$ is strictly singular. 
\end{definition}

\begin{problem}[W. B. Johnson]
\label{ss extension}
Does every infinite dimensional separable Banach space $X$ that is not isomorphic to $c_0$ admit a singular extension $Y$?
\end{problem}

For example, as proved by Kalton and Peck, the space $X=\ell_p$, $1<p<\infty$, admits a singular extension, namely the Banach space $Z_p$ (see \cite[Theorem 6.4, page 27]{KP}). It is not hard to see that if $X$ is HI and it admits a singular extension $Y$ then $Y$ must be HI as well. Therefore, a positive answer to Problem \ref{ss extension} would imply a positive answer of Problem \ref{HI extension}, at least in the separable setting (whether every separable HI space admits a singular extension was also asked by Castillo, Ferenczi, and Gonzalez in \cite[page 4701]{CFG}).

In \cite{CFG} it is proved that the uniformly convex HI space $\mathcal{F}$ constructed by V. Ferenczi in \cite{F} admits an HI short exact sequence $0\to \mathcal{F}\to\mathcal{F}_2\to\mathcal{F}\to0$, such that $\mathcal{F}_2$ is HI, i.e., $\mathcal{F}_2$ is an HI twisted sum of $\mathcal{F}$. This is achieved by using the fact that $\mathcal{F}$ is a space resulting from complex interpolation. Note that in particular $\mathcal{F}_2$ is an HI extension of $\mathcal{F}$.

It has long been an open question whether it is possible to construct a uniformly convex HI space by directly defining a norming set instead of using complex interpolation. It is therefore relevant to ask whether it is possible to find an HI space $X$ and an HI twisted sum $Y$ of $X$ by defining a norming set $W$. The natural approach would be to try to construct a rationally convex norming set $W$ that induces an HI space $Y$ such that
\[\{f|_{\mathbb{N}_{ev}}:f\in W\} = \{f\circ \pi: f\in W\text{ with } \mathrm{supp}(f)\subset\mathbb{N}_{od}\},\]
where $\pi: \mathbb{N}_{ev}\to \mathbb{N}_{od}$ is the order preserving bijection. Then, by taking $X = \overline{\langle\{e_{2n}:n\in\mathbb{N}\}\rangle}$, the Hahn-Banach theorem would yield that $Y/X$ is isometrically isomorphic to $X$.


\begin{thebibliography}{99}
\bibitem{AlK} Albiac, F., Kalton, Nigel J. Topics in Banach space theory. Second edition. With a foreword by Gilles Godefory. Graduate Texts in Mathematics, 233. Springer, 2016. xx+508 pp.

\bibitem{AH} Argyros, S. A.; Haydon, R. G., A hereditarily indecomposable $\mathcal{L}_\infty$-space that solves the scalar-plus-compact problem. Acta Math. 206 (2011), no. 1, 1-54.

  \bibitem{AD}  Argyros, S. A., Deliyanni, I.,
  Examples of asymptotic $\ell_{1}$ Banach spaces. Trans. Amer. Math. Soc. 349 (1997), no. 3, 973--995.
\bibitem{AManStudia}
  Argyros, S. A.; Manoussakis, A., An indecomposable and unconditionally saturated Banach space. Dedicated to Professor Aleksander Pe\l czy\'nski on the occasion of his 70th birthday. Studia Math. 159 (2003), no. 1, 1--32.
\bibitem{AMot} Argyros, S. A.; Motakis, P., A dual method of constructing hereditarily indecomposable Banach spaces. Positivity 20 (2016), no. 3, 625--662.
\bibitem{AMot2} Argyros, S. A.; Motakis, P.,
On the complete separation of asymptotic structures in Banach spaces,
Adv.   in Math., 362 (2020), 106962, 51 pp.
\bibitem{ARaik}   Argyros, S. A., Raikoftsalis, Th., The cofinal
  property of the reflexive indecomposable Banach spaces,
  Ann. Inst. Fourier,  62 (2012) no 1, 1-45.
\bibitem{AFHORSZ}
Argyros, S. A., Freeman, D., Haydon, R., Odell, E., Raikoftsalis,
Th., Schlumprecht Th., Zisimopoulou, D.,
  Embedding uniformly convex spaces into spaces with very few
  operators,
  Journal of Functional Analysis
 262  (2012), no 3, 825-849.
 \bibitem{ATod} Argyros, S. A.; Todorcevic, S. Ramsey methods in analysis. Advanced Courses in Mathematics. CRM Barcelona. Birkh\"auser Verlag, Basel, 2005. viii+257 pp
\bibitem{ATolmemoir}
 Argyros, S. A.; Tolias, A., Methods in the theory of hereditarily indecomposable Banach spaces. Mem. Amer. Math. Soc. 170 (2004), no. 806, vi+114 pp.
 
 \bibitem{CFG}
  Castillo, J. M. F.; Ferenczi, V.; Gonz\'alez, M., Singular twisted sums generated by complex interpolation. Trans. Amer. Math. Soc. 369 (2017), no. 7, 4671-4708.
 
\bibitem{DM} Deliyanni I., Manoussakis A.,
   Asymptotic $\ell_{p}$ hereditarily indecomposable Banach
   spaces. Illinois J. Math. 51 (2007), no. 3, 767-803.
   
   \bibitem{F}
   Ferenczi, V., A uniformly convex hereditarily indecomposable Banach space. Israel J. Math. 102 (1997), 199-225.
   
 \bibitem{G} Gowers, W.T. An infinite Ramsey theorem and some Banach-
   space dichotomies, Annals of Mathematics, 156 (2002), 797--833.
\bibitem{GM}  Gowers, W. T.; Maurey, B. The unconditional basic sequence problem, J. Amer. Math. Soc. 6 (1993), no. 4, 851--874

\bibitem{KP}
Kalton, N. J.; Peck, N. T.,
Twisted sums of sequence spaces and the three space problem.
Trans. Amer. Math. Soc. 255 (1979), 1-30. 

\bibitem{P}
Pe\l czy\'nski, A., Selected problems on the structure of complemented subspaces of Banach spaces. Methods in Banach space theory, 341-354, London Math. Soc. Lecture Note Ser., 337, Cambridge Univ. Press, Cambridge, 2006. 

\bibitem{LT}
  Lindenstrauss, J., Tzafriri, L, Classical Banach spaces. I. Sequence
  spaces. Ergebnisse der Mathematik und ihrer Grenzgebiete,
  Vol. 92. Springer-Verlag, Berlin-New York, 1977.
\bibitem{MR} Maurey, B., Rosenthal H. P., Normalized weakly null sequences with no
unconditional subsequence, Studia Math., 61 (1977), 77-98.
\bibitem{Sch}  Schlumprecht, Th.  An arbitrarily distortable Banach space. Israel J. Math. 76 (1991), no. 1-2, 81--95.
\end{thebibliography}
\end{document}